\documentclass[11pt,reqno]{amsart}
\usepackage{amsmath,amssymb,amsthm,amscd,stmaryrd}
\usepackage{enumitem}
\usepackage[hidelinks]{hyperref}
\numberwithin{equation}{section}
\allowdisplaybreaks[3]
\setlist{topsep=3pt,itemsep=1pt,parsep=0pt,partopsep=0pt}
\AtBeginDocument{%
  \setlength{\abovedisplayskip}{6pt plus 2pt minus 3pt}%
  \setlength{\belowdisplayskip}{6pt plus 2pt minus 3pt}%
  \setlength{\abovedisplayshortskip}{3pt plus 2pt}%
  \setlength{\belowdisplayshortskip}{4pt plus 2pt minus 2pt}%
}

\theoremstyle{plain}
\newtheorem{prop}{Proposition}[section]
\newtheorem{theo}[prop]{Theorem}
\newtheorem{coro}[prop]{Corollary}
\newtheorem{lemm}[prop]{Lemma}

\theoremstyle{definition}
\newtheorem{defi}[prop]{Definition}

\newtheorem{exam}[prop]{Example}
\newtheorem{rema}[prop]{Remark}

\def\cA{{\mathcal A}}
\def\cB{{\mathcal B}}
\def\cG{{\mathcal G}}
\def\cH{{\mathcal H}}
\def\cL{{\mathcal L}}
\def\cM{{\mathcal M}}

\def\cX{{\mathcal X}}
\def\cE{{\mathcal E}}
\def\bfE{{\mathbf E}}

\def\cV{{\mathcal V}}

\def\cY{{\mathcal Y}}
\def\cZ{{\mathcal Z}}
\def\fA{{\mathfrak A}}
\def\fC{{\mathfrak C}}
\def\fG{{\mathfrak G}}
\def\fH{{\mathfrak H}}
\def\fI{{\mathfrak I}}
\def\fK{{\mathfrak K}}
\def\fN{{\mathfrak N}}
\def\bG{{\mathbb G}}
\def\bP{{\mathbb P}}

\def\Br{{\operatorname{Br}}}

\def\ind{{\operatorname{ind}}}

\def\rk{{\operatorname{rk}}}
\def\Spec{{\operatorname{Spec}}}
\def\Hom{{\operatorname{Hom}}}
\def\Aut{{\operatorname{Aut}}}
\def\ra{{\rightarrow}}
\def\GL{{\operatorname{GL}}}
\def\Rep{{\operatorname{Rep}}}

\def\Vect{{\operatorname{Vect}}}

\def\im{{\operatorname{im}}}
\def\Tot{{\operatorname{Tot}}}

\def\rdim{{\operatorname{rdim}}}
\def\ed{{\operatorname{ed}}}

\def\cdim{{\operatorname{cdim}}}
\def\trdeg{{\operatorname{trdeg}}}

\def\Split{{\operatorname{Split}}}
\def\band{{\mathfrak b}}

\def\H{{\mathrm H}}

\title[Essential dimension and faithful rank of finite $p$-gerbes]{Essential dimension and faithful rank of finite $p$-gerbes}
\author{Tianzhi Yang}
\date{\today}
\address{Scuola Normale Superiore, Piazza dei Cavalieri 7, 56126 Pisa, Italy}
\email{tianzhi.yang@sns.it}

\begin{document}

\begin{abstract}
Let $p\ne\operatorname{char}(k)$.  We extend the Karpenko--Merkurjev
theorem from classifying stacks of finite $p$-groups to arbitrary finite
gerbes whose geometric inertia groups are $p$-groups, without assuming
that the gerbe is neutral or that its band is represented by a group
scheme over the base field.  We prove that the essential dimension at
$p$ is exactly the minimum faithful rank obtained after prime-to-$p$
base change, equivalently the faithful rank over a $p$-closure.

We also prove a relative form of the theorem for locally full morphisms
of finite $p$-gerbes: the relative faithful rank equals the supremum of
the essential $p$-dimensions of the fibers.

Finally, we introduce the quotient compression dimension, defined using tame quotient singularities with prescribed fundamental gerbe. For every finite $p$-gerbe $\cG/k$ we
show that its prime local version satisfies
\[
\ed_k(\cG;p)
\le \operatorname{qcdim}_p(\cG)
\le \ed_k(\cG;p)+1.
\]
Thus essential dimension at $p$ determines, up to at most one dimension,
the smallest quotient singularity realizing the gerbe after prime-to-$p$
localization.
\end{abstract}

\maketitle

\section{Introduction}\label{sec:introduction}

Throughout the article, $p$ is a prime and all base fields are assumed to
have characteristic different from $p$, unless explicitly stated otherwise.
All gerbes are algebraic and of finite type over their base field.  We call a
gerbe \emph{finite} if its inertia is finite.  Thus its geometric
automorphism group schemes are finite.  A \emph{finite $p$-gerbe} is a finite
gerbe whose geometric inertia groups are $p$-groups.  Under our standing
assumption $\operatorname{char}(k)\ne p$, the geometric inertia of every finite
$p$-gerbe considered in the main results is finite \'etale.  In the few
auxiliary statements where we work beyond $p$-gerbes, we explicitly assume
that the geometric inertia has order invertible in the base field. 

\subsection{Background and motivation}

Essential dimension measures the minimum number of independent parameters needed to define an object. Let $\mathrm{Field_k}$ be the category of field extensions of $k$, and $F:\mathrm{Field}_k\to \mathrm{Set}$ a covariant functor. For $\xi\in F(K)$, we say $\xi$ descends to $K/L/k$, if there exists $\xi_0\in F(L)$ such that $\xi_0|_K=\xi$. The essential dimension $\ed_k(\xi)$ of $\xi$ is then defined by
\[
\ed_k(\xi)=
\min_{\xi\text{ descends to }L}\trdeg_k(L).
\]
The essential dimension of $F$ is obtained by taking the supremum over all fields and all objects.

This notion was initially introduced by Buhler and Reichstein for finite groups \cite{BuhlerReichstein}; the general functorial formulation is due to Merkurjev, as presented in \cite{BerhuyFavi}. Brosnan--Reichstein--Vistoli subsequently defined essential dimension for algebraic stacks \cite[Definition~2.2]{BRV}. 

A basic illustration is the calculation of Brosnan--Reichstein--Vistoli
for the moduli stacks of pointed smooth and stable curves, including the
hyperelliptic locus; for $\cM_{g,n}$, away from several low genus
exceptions, one obtains the expected value $3g-3+n$. The appendix by
Fakhruddin computes the essential dimension of moduli stacks of abelian
varieties and principally polarized abelian varieties
\cite[Theorem~1.2, Theorem~7.2, and Theorem~A.1]{BRV}.  A second structural result is
the genericity theorem.  If $\cX$ is a smooth integral tame stack,
locally of finite type over $k$, with coarse moduli space $M$, then its
pullback to $k(M)$ is the generic gerbe $\cX_{k(M)}$, and
\[
\ed_k(\cX)
=\dim(M)+\ed_{k(M)}(\cX_{k(M)});
\]
see \cite[Theorem~6.1]{BRV}. The genericity theorem is then generalized to regular integral weakly tame stack by G.~Bresciani and A.~Vistoli
\cite[Introduction and Theorem~1.1]{BrescianiVistoliGeneric}.  Thus, for a
broad class of tame stacks, the computation of essential dimension is
reduced to understanding the essential dimension of gerbes. 

For the classifying stack of a constant $p$-group $P$, the decisive model is the theorem of Karpenko and Merkurjev: if $\operatorname{char}(k)
\ne p$, and $\mu_p\subset k$, then
\[
\ed_k(BP)=\ed_k(BP;p)=\rdim_k(P),
\]
where $\rdim_k(P)$ is the least dimension of a faithful
$k$-representation \cite[Theorem~4.1]{KM}. This result was subsequently extended by L\"otscher--MacDonald--Meyer--Reichstein to
possibly twisted finite $p$-groups, which becomes constant over a Galois extension of $p$-power degree; see \cite[Theorem~7.1]{LMMRTori}. These results still concern neutral gerbes. 

L\"otscher later developed the non-neutral direction for gerbes
banded by finite diagonalizable groups. More precisely, if a gerbe
$\cG$ is banded by $A=\ker(Q\to S)$, where $Q$ is an invertible torus and
$S$ is a split torus, he proves
\[
\ed(\cG)=\cdim(\cG)+\ed(A),
\]
where $\cdim(\cG)$ is the canonical dimension of $\cG$ \cite[Theorem~1.1]{LotscherGerbes}. For a finite diagonalizable group $A$, the quantity $\ed(A)$ is the smallest number of characters required to generate its character group, or, equivalently, the dimension of a minimal faithful representation of $A$; in this way, the computation inherently involves the representation rank. This also shows that, in the case of non-neutral gerbes, the Karpenko–Merkurjev equality involves contributions of an arithmetic nature.

It is crucial, however, to distinguish this invariant of the band from the one used in the present article: the number $\rdim(A)$ for a group scheme $A$ and the faithful rank $\rdim(\cG)$ of a gerbe—introduced below via vector bundles on $\cG$—are distinct concepts and, in general, do not coincide for a non-neutral gerbe.

To the best of the author's knowledge, the existing literature stops at
these settings.  In particular, there has not been a general analogue of
the Karpenko--Merkurjev equality for non-neutral finite $p$-gerbes with
possibly nonabelian bands.

\subsection{From groups to finite gerbes}

Passing from classifying stacks to arbitrary finite gerbes is not a formal generalization of the preceding results. If a gerbe is neutral, or if the gerbe is abelian, then its band is represented by a group scheme $G$ over the base field, and vector bundles on $\cG$ may be studied through representations of $G$. For a non-neutral gerbe, after choosing an object over a field extension one obtains an automorphism group there, but its descent is governed only by the band; when the band is nonabelian, the descent datum is genuinely outer and need not arise from a group scheme over $k$.

This phenomenon is quite common in practice. Non-neutral gerbes occur naturally as generic or residue gerbes. For instance, gerbes arise as the stack that classifies twisted forms of varieties, and a variety admits a model over its field of moduli precisely when its associated residue gerbe is neutral; see \cite{Bresciani_Vistoli_2024}. Numerous examples are known of varieties or algebraic structures that fail to be defined over their fields of moduli; see, for example, \cite{bresciani-plane,Bresciani-Divisor,yang2026fieldsmodulismoothdel}.

\subsection{The main theorems}
It is therefore natural to seek an approach that is intrinsic to the gerbe and does not require a presentation as a classifying stack. The linear objects that remain available in this generality are vector bundles on the gerbe itself. Every $\cE\in\Vect(\cG)$ carries an action of geometric inertia on its
fibers. Following
\cite[Definition~3.0.4]{BraggLieblich}, define
\[
\rdim(\cG)=
\min\{\rk(\cE):\cE\in\Vect(\cG)
       \text{ is faithful on geometric inertia}\}
\]
and
\[
\rdim_p(\cG)=
\min_{\substack{L/k\ {\rm finite}\\p\nmid[L:k]}}\rdim(\cG_L).
\]
It is appropriate to remark that, the use of $\Vect(\cG)$ is part of a broader program of understanding
gerbes through their vector bundles.  In earlier joint work with
G.~Bresciani, we studied neutrality through $\Vect(\cG)$: vector bundles of
prescribed geometric type were used to determine whether $\cG$ is neutral,
leading in particular to a classification of so called \emph{neutral
representations} in dimension at most three
\cite{BrescianiYangNeutral}.  The present paper develops a complementary
quantitative question: how small can a vector bundle be while still
faithfully detecting the geometric inertia, and how is this minimum related
to essential dimension?

The first main theorem gives a complete prime local answer.

\begin{theo}\label{thm:intro-main}
Let $\operatorname{char}(k)\ne p$, and let $\cG/k$ be a finite gerbe
whose geometric inertia groups are $p$-groups.  If $k^{(p)}$ is a
$p$-closure of $k$, then
\[
\ed_k(\cG;p)
=\rdim_p(\cG)
=\rdim(\cG_{k^{(p)}}).
\]
If $k$ is $p$-special, then
\[
\ed_k(\cG)=\ed_k(\cG;p)=\rdim(\cG).
\]
More generally, the same absolute equality holds if
$\mu_p\subset k$ and the band of $\cG$ becomes constant over a finite
Galois extension of $p$-power degree.
\end{theo}

The prime local formulation is necessary: over a general field, an
unconditional equality with $\rdim(\cG)$ would be false even for neutral
gerbes; see Remark~\ref{rem:localization-necessary}.

Viewed against the results recalled above, Theorem~\ref{thm:intro-main} completes the remaining finite $p$-gerbe cases. The mechanism behind this result is a non-neutral index theorem. Let $\cA/k$ be a finite
gerbe whose geometric inertia groups have order invertible in $k$, without
assuming that they have $p$-power order, let $\fC$ be a split finite
diagonalizable central subband, and put $\cH=\cA\sslash\fC$, the rigidification of $\cA$ by $\fC$. For
$\chi\in\fC^*$, let $I_\chi(\cA)$ be the ideal generated by the ranks of
vector bundles of pure central character $\chi$.  Let
$\eta:\Spec K\to\cH$ be the generic frame object constructed in
Lemma~\ref{lem:generic-object-rigidification}; its fiber has pushout along
$\chi$ defining a Brauer class $\beta_\eta(\chi)\in\Br(K)$.

\begin{theo}[Non-neutral index theorem]\label{thm:intro-index}
In the preceding notation,
\[
I_\chi(\cA)=\ind\beta_\eta(\chi)\mathbb Z
\quad\text{for every }\chi\in\fC^*.
\]
\end{theo}

The proof of Theorem~\ref{thm:intro-main} follows a strategy that blends
the argument of Karpenko--Merkurjev with the split central subgroup method
of L\"otscher--MacDonald--Meyer--Reichstein, modified as necessary to deal
with non-neutral gerbes.  Concretely, the computation for the
$(\mu_p)^c$-gerbe is taken from \cite[Theorems~2.1 and~3.1]{KM}, whereas the
split central subband and normal subband arguments are gerbe theoretic
analogues of the corresponding constructions in \cite{LMMRTori}.  The main
new feature is that these steps have to be performed at the level of a
band, instead of a global group scheme, and then assembled using
Theorem~\ref{thm:intro-index}.

\subsection{Relative theorem and applications}

Theorem~\ref{thm:intro-main} has a relative form.  If
$f:\cX\to\cY$ is a locally full morphism of finite $p$-gerbes, write
$I_{\cX/\cY}$ for its relative inertia and define
\[
\rdim(f)=\min\{\rk(\cE):
I_{\cX/\cY}\to\GL(\cE)\text{ is faithful}\}.
\]
Define $\rdim_p(f)$ after prime-to-$p$ localization, and let
\[
\ed_k(f;p)=
\sup_{K/k,\,y\in\cY(K)}
\ed_K(\cX_y;p).
\]
\begin{theo}
    With the notations as above, we have
\[
\ed_k(f;p)=\rdim_p(f).
\]
\end{theo}
After passage to a $p$-closure, the common value is realized by an actual
fiber whose ordinary essential dimension, essential $p$-dimension, and
faithful rank all coincide.

For a surjection of finite $p$-groups $P'\to P$ with kernel $N$, this says
that the largest essential dimension of a lifting gerbe for a $P$-torsor is
the least dimension of a $P'$-representation whose restriction to $N$ is
faithful; the generic $P$-torsor obtained from a free open subset of a
representation attains the maximum.  This answers, for finite
$p$-groups, a question raised by L\"otscher for arbitrary normal subgroups
\cite[Remark~4.4]{LotscherFiber}.

A common mechanism underlies both the absolute and relative results.
Central characters and Brauer indices provide a way to pass from the
geometry of a non-neutral gerbe to concrete rank constraints on vector
bundles.  In the absolute case these constraints compute essential
$p$-dimension, while in the relative case they compute the largest
essential $p$-dimension among the fibers.  This framework is intrinsic to
the gerbe and continues to apply when there is no global group scheme as its band.

\subsection{Quotient compression dimension}
We conclude with a geometric interpretation.  For a finite gerbe
$\cG/k$ whose geometric inertia has order invertible in $k$, we introduce the quotient compression dimension $\operatorname{qcdim}_k(\cG)$,
the least dimension of a tame quotient singularity whose fundamental gerbe
is $\cG$.  It is equal to the minimum rank of a faithful vector bundle on
$\cG$ with no pseudoreflections.  In particular, for finite $p$-gerbes,
Corollary~\ref{thm:qcdim-ed} gives the following geometric consequence of
Theorem~\ref{thm:intro-main}:
\[
\ed_k(\cG;p)
\le
\operatorname{qcdim}_p(\cG)
\le
\ed_k(\cG;p)+1.
\]
Thus, after prime-to-$p$ localization, the essential dimension determines
up to one the smallest dimension of a tame quotient singularity retaining
$\cG$ as its fundamental gerbe.

\subsection{Structure of the paper}

Section~\ref{sec:essential-dimension-background} recalls the necessary
background on essential dimension.  Section~\ref{sec:bands} fixes the
descent formalism for bands, subbands, and kernel subbands.  In
Section~\ref{sec:central-index} we prove the non-neutral index theorem.
Section~\ref{sec:gerby-km} constructs the canonical split central subband,
proves attainment of the relevant rank ideals, and establishes the main
theorem under $p$-primary descent.  Section~\ref{sec:p-localization}
removes the $p$-primary descent hypothesis by prime-to-$p$ localization.
Section~\ref{sec:relative-rank} develops relative faithful rank and applies
the relative theorem to lifting problems.  Finally,
Section~\ref{sec:compression-singularities} introduces quotient compression
dimension.

\section{Essential dimension and faithful rank}
\label{sec:essential-dimension-background}

The material in this section is standard, and no new result is claimed.
We follow \cite[Sections~1 and~2]{BerhuyFavi} for the functorial
definition and its specialization to algebraic groups,
\cite[Sections~2 and~4]{BRV} for algebraic stacks and gerbes, and
\cite[Sections~1 and~3]{LMMRTori} for essential dimension at a prime and
$p$-closures.

\subsection{Fields of definition}

Let $\mathcal X$ be a category fibered in groupoids over $k$, and let
$K/k$ be a field extension.  An object $\xi\in\mathcal X(K)$ is said to \emph{descend} to an intermediate field $k\subset K_0\subset K$ if there exists an object $\xi_0\in\mathcal X(K_0)$ together with an isomorphism
\[
\xi_0|_K\simeq\xi.
\]
The essential dimension of $\xi$ over $k$ is
\[
\ed_k(\xi)=
\min_{\xi\text{ descends to }K_0}\trdeg_k(K_0).
\]
The essential dimension of $\mathcal X$ is obtained by taking the
supremum over all fields and all objects:
\[
\ed_k(\mathcal X)=
\sup_{\substack{K/k\\ \xi\in\mathcal X(K)}}\ed_k(\xi).
\]
Equivalently, one may apply the functorial definition to the functor
sending $K$ to the set of isomorphism classes in $\mathcal X(K)$.
For algebraic stacks that are locally of finite presentation, each individual
object can be defined over some finitely generated field extension, even though
the supremum taken over all such objects may still be infinite.  By contrast,
when $X$ is an algebraic space of finite type over $k$, we have
$\ed_k(X)=\dim(X)$ \cite[Example~2.4]{BRV}.  In particular,
$\ed_k(\bP^n)=n$.

For a prime $p$, one is allowed first to make a finite extension of
degree prime to $p$.  Thus
\[
\ed_k(\xi;p)=
\min_{\substack{K'/K\ {\rm finite}\\p\nmid[K':K]}}
\ed_k(\xi_{K'}),
\quad
\ed_k(\mathcal X;p)=
\sup_{\substack{K/k\\ \xi\in\mathcal X(K)}}\ed_k(\xi;p).
\]
This invariant ignores phenomena which disappear after prime-to-$p$
extensions.  It always satisfies
\[
\ed_k(\mathcal X;p)\le\ed_k(\mathcal X).
\]

\subsection{Groups and classifying stacks}

Let $\fG$ be an affine group scheme over $k$.  The groupoid
$B\fG(K)$ is the groupoid of $\fG_K$-torsors over $\Spec K$;
consequently its set of isomorphism classes is $\H^1(K,\fG)$ in the
fppf topology.  By definition,
\[
\ed_k(\fG)=\ed_k(B\fG),
\quad
\ed_k(\fG;p)=\ed_k(B\fG;p).
\]
Thus essential dimension of a group scheme measures the number of
parameters needed to define all of its torsors.  If $V$ is a generically
free linear representation of $\fG$, then
\[
\ed_k(B\fG)\le\dim(V)-\dim(\fG),
\]
see \cite[Lemma~4.11]{BerhuyFavi}.  In particular, if $\fG$ is finite constant
and $V$ is faithful, then the action on $V$ is generically free, and hence
\[
\ed_k(B\fG)\le\dim(V).
\]
More precisely, the Karpenko--Merkurjev theorem states that if $P$ is a
finite constant $p$-group and $\mu_p\subset k$, then
\[
\begin{aligned}
\ed_k(BP)&=\ed_k(BP;p)\\
&=\min\{\rk(V): V\text{ is a faithful $k$-representation of }P\};
\end{aligned}
\]
see \cite[Theorem~4.1]{KM}. Thus the preceding
upper bound becomes an equality for a faithful representation of minimal
rank.

\subsection{Gerbes}

A gerbe $\cG/k$ is a stack that is locally nonempty, and any two of its objects are
locally isomorphic.  If it has a $k$-rational point $x\in\cG(k)$, then it is called
\emph{neutral} and
\[
\cG\simeq B\underline{\Aut}_k(x).
\]
Its essential dimension is then the essential dimension of this
classifying stack.  A non-neutral gerbe has no such global presentation,
but the definition of $\ed_k(\cG)$ is unchanged: it is the supremum of
the essential dimensions of objects $\xi\in\cG(K)$ over all extensions
$K/k$.  Both the field needed to obtain an object and the parameters
needed to define its twisted forms are part of this invariant.

For the finite gerbes used in this paper, a vector bundle is called
\emph{faithful} if its inertia action on every geometric fiber is faithful.
We make the following definition; see \cite[Definition~3.0.4]{BraggLieblich}.

\begin{defi} Let $\cG/k$ be a finite gerbe and $\Vect(\cG)$ the category of vector bundles over $\cG$. Define
    \[
\rdim(\cG)=
\min\{\rk(\cE):\cE\in\Vect(\cG)
       \text{ is faithful}\}
\]
and
\[
\rdim_p(\cG)=
\min_{\substack{L/k\ {\rm finite}\\p\nmid[L:k]}}\rdim(\cG_L).
\]
\end{defi}

After passing to an algebraic closure, a faithful inertia
representation is therefore a faithful representation of a finite constant
group.  The complement of its free locus is the finite union of the fixed
subspaces of the nonidentity elements, so the free locus is nonempty and
open.  Applied fiberwise to a faithful vector bundle, this gives
\[
\ed_k(\cG)\le\rdim(\cG),
\]
as proved below in Proposition~\ref{prop:ed-upper-main}.  The difficulty is
the reverse inequality.  In the neutral case it is a problem about
representations of a global group scheme.  In the non-neutral case only the
band, with its outer descent, is globally available; the gerbe class
contributes an additional twisting obstruction.

Finally, a $p$-closure $k^{(p)}$ is a $p$-special algebraic extension
obtained, up to the usual purely inseparable completion, from a Sylow
pro-$p$ subgroup of the absolute Galois group.  It is a filtered union
of finite extensions of $k$ of degree prime to $p$, and prime-local
essential dimension over $k$ agrees with essential dimension at $p$ after
base change to $k^{(p)}$; see Lemma~\ref{lem:p-special-ed-invariance} below.

\section{Bands}\label{sec:bands}

We recall the part of the standard theory of bands and outer descent
needed below.  Our conventions follow Giraud
\cite[Chapter~IV, Sections~1.1 and~2.2]{Giraud}; for a concise modern
treatment in the setting of finite \'etale gerbes, see
\cite[Section~3.1]{JavanpeykarLoughran}.

Most of the material in this section is standard, and we record it mainly to
fix conventions used later.  The kernel subband of
Lemma~\ref{lem:kernel-subband} is new here, but its construction is an
elementary consequence of outer descent.

\subsection{The stack of bands}

Let $\mathcal C$ be a site.  If $U\in\mathcal C$, let $\fG$ and
$\fH$ be sheaves of groups on the localized site $\mathcal C/U$.
For a section $h$ of $\fH$, write $\operatorname{Int}(h)$ for the
inner automorphism $x\mapsto hxh^{-1}$.
Two homomorphisms $f,g:\fG\to\fH$ are called
\emph{locally conjugate} if, after passing to a covering
$\{U_a\to U\}$, there are sections $h_a\in\fH(U_a)$ such that
\[
f|_{U_a}=\operatorname{Int}(h_a)\circ g|_{U_a}.
\]
The sheaf of \emph{outer homomorphisms} is defined by
\[
\underline{\Hom}_{\mathrm{out}}(\fG,\fH)
:=
\bigl(
\underline{\Hom}_{\mathrm{grp}}(\fG,\fH)/\fH
\bigr)^{\#}.
\]
So an outer homomorphism $\fG\to\fH$ is a locally defined homomorphism modulo local
conjugacy.

Composition is well defined.  Indeed, if
$f,f':\fG\to\fH$ and $q,q':\fH\to\fK$ satisfy
$f'=\operatorname{Int}(h)\circ f$ and
$q'=\operatorname{Int}(k)\circ q$ for local sections $h$ of $\fH$
and $k$ of $\fK$, then
\[
q'\circ f'
=
\operatorname{Int}\bigl(kq(h)\bigr)\circ q\circ f.
\]
Thus group sheaves on $\mathcal C/U$, with outer homomorphisms as
morphisms, form a category.  Restriction along $V\to U$ makes these
categories into a fibered category
$\operatorname{PreBand}_{\mathcal C}\to\mathcal C$.  Because the
outer-Hom sets were sheafified, morphisms already satisfy descent; objects
need not.

\begin{defi}[Giraud's stack of bands]\label{def:stack-of-bands}
The \emph{stack of bands} on $\mathcal C$, denoted
$\operatorname{Band}_{\mathcal C}$, is the stackification of
$\operatorname{PreBand}_{\mathcal C}$.  A \emph{band over $U$} is an
object of the fiber category $\operatorname{Band}_{\mathcal C}(U)$.
This is Giraud's definition; see
\cite[Chapter~IV, D\'efinition~1.1.6]{Giraud}.
\end{defi}

In other words, a band over $U$ can be presented by
\begin{enumerate}[label=\textup{(\roman*)}]
\item a covering $\{U_i\to U\}$;
\item a group sheaf $\fG_i$ on each $\mathcal C/U_i$; and
\item outer isomorphisms
\[
\overline\lambda_{ij}:
\fG_j|_{U_{ij}}\xrightarrow{\sim}\fG_i|_{U_{ij}}
\]
which satisfy
$\overline\lambda_{ij}\overline\lambda_{jk}
=\overline\lambda_{ik}$ on triple overlaps.
\end{enumerate}
Choosing actual representatives $\lambda_{ij}$ of the
outer isomorphisms changes the displayed cocycle condition into
\[
\lambda_{ij}\lambda_{jk}
=\operatorname{Int}(g_{ijk})\lambda_{ik}
\]
for suitable local sections $g_{ijk}\in\fG_i$.  Thus a band is a group
sheaf glued only up to inner automorphism.  Stackification is necessary
precisely because such outer descent data need not come from a global
group sheaf.

In this paper $\mathcal C$ is the fppf site of a scheme $S$, which has
final object $S$; by a band over $S$ we mean an object of
$\operatorname{Band}_{\mathcal C}(S)$.

\begin{rema}
If all $\fG_i$ are abelian, inner automorphisms disappear and ordinary
descent is recovered. So an abelian band is an abelian group scheme \cite[Chapter~IV, Proposition~1.2.3]{Giraud}.
\end{rema}

\subsection{The band of a gerbe}

Let $\cG\to S$ be a gerbe for the fppf topology.  Choose a covering
$\{U_i\to S\}$ and objects $x_i\in\cG(U_i)$.  After refining the
covering on the overlaps, choose isomorphisms
\[
\phi_{ij}:x_j|_{U_{ij}}\ra x_i|_{U_{ij}},
\quad U_{ij}=U_i\times_S U_j,
\]
and put
\[
\fG_i=\underline{\Aut}_{U_i}(x_i).
\]
Conjugation by $\phi_{ij}$ gives
\[
\lambda_{ij}:\fG_j|_{U_{ij}}\ra \fG_i|_{U_{ij}},
\quad
h\longmapsto\phi_{ij}h\phi_{ij}^{-1}.
\]
On a triple overlap, define
\[
g_{ijk}=\phi_{ij}\phi_{jk}\phi_{ik}^{-1}\in\fG_i(U_{ijk}).
\]
Then
\begin{equation}\label{eq:band-outer-cocycle}
\lambda_{ij}\lambda_{jk}
=\operatorname{Int}(g_{ijk})\lambda_{ik}.
\end{equation}
If $\phi_{ij}$ is replaced by $a_{ij}\phi_{ij}$, where $a_{ij}\in \fG_i(U_{ij})$, then
$\lambda_{ij}$ is replaced by
$\operatorname{Int}(a_{ij})\lambda_{ij}$.  Thus the $\fG_i$ do not in
general descend as a sheaf of groups: their transition maps satisfy the
cocycle condition only modulo inner automorphisms.

\begin{defi}\label{def:band}
The \emph{band} $\operatorname{Band}(\cG)$ of $\cG$ is the object of
$\operatorname{Band}_{\mathcal C}(S)$ represented by the local groups
$\fG_i$ and the outer transition isomorphisms defined above. If $\mathfrak l$ is a band over $S$, an
$\mathfrak l$-\emph{banding} of $\cG$ is an isomorphism
$\mathfrak l\xrightarrow{\sim}\operatorname{Band}(\cG)$.
\end{defi}

\begin{defi}\label{def:represented-band}
For a group scheme $\fG\to S$, write
\[
\operatorname{band}(\fG)=\operatorname{Band}(B\fG).
\]
A band $\mathfrak l$ over $S$ is \emph{represented by $\fG$} if there
is an isomorphism
$\mathfrak l\simeq\operatorname{band}(\fG)$.  If $S=\Spec k$ and
$\fG=P_k$ for an abstract finite group $P$, we say that
$\mathfrak l$ is \emph{constant}.  We say that an extension $K/k$
\emph{splits} a finite band if its base change to $K$ is constant.
\end{defi}

Representability of the band and neutrality of the gerbe must not be
confused.  If $x\in\cG(S)$, then
\[
\cG\simeq B\underline{\Aut}_S(x),
\]
so neutrality implies that the band is represented.  The converse is
false: a non-neutral gerbe may be banded by an honest group scheme $\fG$.

Although a band need not be represented, its center always is; see also
\cite[Part~II, Section~5.1]{Duskin}.

\begin{prop}\label{prop:center-of-band}
Let $\mathfrak l$ be a band over $S$. The centers of its local group
representatives descend canonically to a sheaf of abelian groups
$Z(\mathfrak l)$.  Over a field, if the local representatives are
finite group schemes, then $Z(\mathfrak l)$ is a finite commutative
group scheme.  It is \'etale whenever the geometric inertia has order
invertible in the field; in particular this holds for the finite
$p$-gerbes considered below when $\operatorname{char}(k)\ne p$.
\end{prop}

\begin{proof}
Every isomorphism $\fG_j\to\fG_i$ restricts to an isomorphism
$Z(\fG_j)\to Z(\fG_i)$.  The ambiguity in the transition maps is inner,
and inner automorphisms act trivially on the center.  Therefore
\eqref{eq:band-outer-cocycle} restricts to an ordinary cocycle on the
centers, which gives effective descent.  Over a field the descended
center is finite.  If the inertia order is invertible, the local
representatives are finite \'etale, and so is their center.
\end{proof}

\subsection{Subbands and inertia representations}
We first make precise the notion of a subband.

\begin{defi}\label{def:subband}
Let $\mathfrak h$ and $\mathfrak l$ be bands over $S$.  A
\emph{subband} $\mathfrak h\subset\mathfrak l$ is a morphism of bands
which, locally on $S$ and after choosing group representatives, is
induced by a monomorphism of group sheaves
\[
\fK_i\lhook\joinrel\ra\fG_i.
\]
Thus a subband is what is usually called an injective morphism of bands;
compare \cite[III,~3.3.3, p.~205]{Saavedra}.  It is a \emph{normal
subband} if every local image is normal in $\fG_i$, and a \emph{central
subband} if every local image is contained in $Z(\fG_i)$.
\end{defi}

\begin{lemm}
\label{lem:subband-descent}
Present $\mathfrak l$ on a cover $\{U_i\to S\}$ by group sheaves
$\fG_i$ and outer transition isomorphisms
\[
\overline\lambda_{ij}:\fG_j|_{U_{ij}}
\xrightarrow{\sim}\fG_i|_{U_{ij}}.
\]
For each $i$, let $\iota_i:\fK_i\hookrightarrow\fG_i$ be a
monomorphism.  These monomorphisms define a subband of $\mathfrak l$ if
and only if, after refining the cover, there are outer isomorphisms
\[
\overline\mu_{ij}:\fK_j|_{U_{ij}}
\xrightarrow{\sim}\fK_i|_{U_{ij}}
\]
such that
\begin{enumerate}[label=\textup{(\roman*)}]
\item
$\overline\mu_{ij}\overline\mu_{jk}=\overline\mu_{ik}$ on
$U_{ijk}$; and
\item
\[
\overline\iota_i\,\overline\mu_{ij}
=
\overline\lambda_{ij}\,\overline\iota_j
\quad\text{in }
\underline{\Hom}_{\mathrm{out}}(\fK_j,\fG_i)(U_{ij}).
\]
\end{enumerate}
\end{lemm}

\begin{proof}
Conditions \textup{(i)} and \textup{(ii)} are precisely the descent conditions for the source band and, respectively, for its morphism to $\mathfrak l$.
\end{proof}

Suppose that a descended normal subband
$\mathfrak h\subset\mathfrak l$ is represented locally by
$\fK_i\subset\fG_i$, and that a central subband
$\fC\subset Z(\mathfrak l)$ is represented locally by
$\fC_i\subset Z(\fG_i)$.  Then the intersections
$\fK_i\cap\fC_i$ define a subband $\mathfrak h\cap\fC$, independently of
the presentation.  Normality makes the $\fK_i$ stable under the
conjugacies in the outer descent data, while conjugation acts trivially on
the central $\fC_i$; hence the intersections inherit compatible transition
maps satisfying the two conditions of Lemma~\ref{lem:subband-descent}.

We now describe how vector bundles fit into the band formalism. Denote by
$\fI_{\cG}\to\cG$ the inertia group stack. Any vector bundle $\cE$ on
$\cG$ induces a canonical inertia representation
\[
\rho_{\cE}:\fI_{\cG}\ra\operatorname{GL}_{\cG}(\cE).
\]
More explicitly, for an object $x\in\cG(T)$, the functoriality of $\cE$ yields
\[
\rho_{\cE,x}:\underline{\Aut}_T(x)\ra
\operatorname{GL}_T(x^*\cE).
\]
Given an isomorphism $\phi:x\to y$, the induced map $x^*\cE\to y^*\cE$
conjugates $\rho_{\cE,x}$ to $\rho_{\cE,y}$, once we identify
$\underline{\Aut}_T(x)$ with $\underline{\Aut}_T(y)$ via
$g\mapsto\phi g\phi^{-1}$. Consequently, these local inertia
representations glue compatibly with the outer descent data that define
$\operatorname{Band}(\cG)$.

Equivalently, if $\cE$ has rank $n$, its classifying morphism
\[
c_{\cE}:\cG\ra B\operatorname{GL}_{n,S}
\]
induces a morphism of bands
\begin{equation}\label{eq:induced-band-morphism}
\overline\rho_{\cE}:
\operatorname{Band}(\cG)\ra
\operatorname{band}(\operatorname{GL}_{n,S}).
\end{equation}
Locally, $\overline\rho_{\cE}$ is represented by the homomorphisms
$\rho_{\cE,x}$ above.  We refer to
\eqref{eq:induced-band-morphism} simply as the morphism of bands induced
by $\cE$.  Since a band need not be a global group scheme, this morphism
need not be represented by a global homomorphism of group sheaves.

The preceding discussion also makes precise how a represented central
subband acts on vector bundles.

\begin{lemm}\label{lem:central-subband-inertia}
Let $\cG\to S$ be a gerbe with band $\band$, and let
$\fC\subset Z(\band)$ be a central subband represented by a group
scheme $\fC\to S$.  Then $\fC$ determines canonically a central
monomorphism
\[
\fC\times_S\cG\lhook\joinrel\ra\fI_{\cG}.
\]
For an object $x\in\cG(T)$, its pullback is the inclusion
$\fC_T\hookrightarrow Z(\underline{\Aut}_T(x))$ associated with the
band.  Consequently $\fC$ acts canonically on every vector bundle on
$\cG$.
\end{lemm}

\begin{proof}
For $x\in\cG(T)$, an identification of $\band_T$ with the band of
$\underline{\Aut}_T(x)$ identifies $Z(\band)_T$ with
$Z(\underline{\Aut}_T(x))$.  This identification is canonical: changing
the identification of the band changes the corresponding automorphism
group by an inner automorphism, and inner automorphisms act trivially on
the center.  The inclusions
$\fC_T\hookrightarrow Z(\underline{\Aut}_T(x))\hookrightarrow
\underline{\Aut}_T(x)$ are therefore functorial in $x$ and compatible
with base change.  They assemble to the asserted central monomorphism into
the inertia stack, and the inertia action then gives the asserted action on
every vector bundle.
\end{proof}

\begin{lemm}\label{lem:kernel-subband}
Let $\cE$ be a vector bundle on a gerbe $\cG$, and put
$\band=\operatorname{Band}(\cG)$.  The local kernels
\[
\fN_x=\ker(\rho_{\cE,x})\triangleleft\underline{\Aut}(x)
\]
determine a normal subband $\band_{\cE}\subset\band$.  Moreover, $\cE$ is
faithful on geometric inertia if and only if $\band_{\cE}=1$.  If
$\fC\subset Z(\band)$ is a central subband, then the kernel of the
restricted action of $\fC$ on $\cE$ is $\band_{\cE}\cap\fC$.
We call $\band_{\cE}=\ker(\overline\rho_{\cE})$ the kernel subband of $\cE$.
\end{lemm}

\begin{proof}
An isomorphism $\phi:x\to y$ carries $\fN_x$ to $\fN_y$ by
conjugation, because it intertwines the corresponding inertia
representations.  These kernels therefore satisfy the descent criterion
of Lemma~\ref{lem:subband-descent}; they are normal because they are
kernels of group homomorphisms.  This gives $\band_{\cE}$, which is trivial exactly when all geometric inertia
representations are faithful.  Finally, restriction of
$\rho_{\cE,x}$ to $\fC$ has kernel $\fN_x\cap\fC$, which gives the last
assertion after descent.
\end{proof}

Consequently the notation $\band$, $Z(\band)$, $\fC(\band)$, and the
terms \emph{normal subband} and \emph{kernel subband} used below are
independent of the chosen cover, local group representatives, and actual
representatives of the outer transition maps; none presupposes a global
group scheme representing $\band$.

\section{The non-neutral index theorem}
\label{sec:central-index}

Let $\cG/k$ be a finite gerbe whose geometric inertia groups have order
invertible in $k$, and let $\fC$ be a split finite diagonalizable central
subband of $\operatorname{Band}(\cG)$.  In particular, the geometric inertia is
finite \'etale and the order of $\fC$ is invertible in $k$.  No $p$-group
hypothesis is imposed in this section.

The proof of the main theorem~\ref{thm:nonneutral-central-index} follows the approach of Karpenko and Merkurjev; see \cite[Theorem~4.4 and Remark~4.5]{KM}. We keep their overall framework, but reformulate each step so that it is intrinsic to the gerbe. Concretely, we replace the quotient group scheme with the rigidification $\cG\sslash\fC$, the generic quotient torsor with a generic frame object for this rigidification, and the representation category of a global group scheme with the character summands of $G_0(\cG)$. This extra work is exactly what allows the argument to go through when $\cG$ is non-neutral and its band has no representative over $k$.

\subsection{Central characters and rank ideals}

Put $\fC^*=\Hom(\fC,\mathbb G_m)$.  Since $\fC$ is split
diagonalizable, the equivalence between modules with a
$\fC$-action and $\fC^*$-graded modules
\cite[Expos\'e~I, Proposition~4.7.3]{SGA3I} gives every coherent sheaf over $\cG$ a
functorial decomposition
\[
\mathcal F=\bigoplus_{\chi\in\fC^*}\mathcal F_\chi.
\]
Let $\operatorname{Coh}_\chi(\cG)$ be the exact category of coherent
sheaves of pure character $\chi$, and set
\[
G_0^\chi(\cG)=K_0(\operatorname{Coh}_\chi(\cG)).
\]
Here $K_0$ denotes the Grothendieck $K$-group of the exact category
$\operatorname{Coh}_\chi(\cG)$.
After a finite separable extension which neutralizes $\cG$ and makes its
inertia constant, the gerbe becomes $BP$ for a finite constant group $P$
whose order is invertible in the field.  Coherent sheaves then become
finite-dimensional representations of $P$, and are therefore locally free
over the gerbe.  Rank defines a homomorphism
\[
\rk_\chi:G_0^\chi(\cG)\ra\mathbb Z.
\]

\begin{defi}\label{def:rank-ideal}
The \emph{$\chi$-rank ideal} and its positive generator are
\[
I_\chi(\cG)=\operatorname{im}(\rk_\chi)
=d_\chi(\cG)\mathbb Z.
\]
Equivalently,
\[
d_\chi(\cG)=
\gcd\{\rk(\cE):0\ne \cE\in\Vect_\chi(\cG)\}.
\]
\end{defi}

The distinction between $d_\chi$ and the minimum positive block rank is
important.  A greatest common divisor need not be realized by an actual vector bundle.

\begin{lemm}\label{lem:semisimple-gerbe}
The category $\Vect(\cG)$ is a Hom-finite semisimple $k$-linear category,
and each $\Vect_\chi(\cG)$ is a semisimple direct summand.
\end{lemm}

\begin{proof}
Choose a finite separable extension $L/k$ which neutralizes $\cG$ and
makes its inertia constant.  Then
\[
\cG_L\simeq BP
\]
for a finite constant group $P$ with $\operatorname{char}(L)\nmid |P|$.
Hence
\[
\Vect(\cG_L)\simeq\operatorname{Rep}_L(P),
\]
and Maschke's theorem shows that this category is semisimple.

We now descend a splitting explicitly. Let
\[
0\ra\cE'\ra\cE
\xrightarrow{q}\cE''\ra0
\]
be exact on $\cG$.  Since $\cG$ is finite over $k$, the spaces
$\Hom_{\cG}(\cE'',\cE)$ and
$\operatorname{End}_{\cG}(\cE'')$ are finite-dimensional, and flat base change gives
\[
\Hom_{\cG}(\cE'',\cE)\otimes_kL
\simeq
\Hom_{\cG_L}(\cE''_L,\cE_L),
\]
and similarly for endomorphisms. Concretely, these assertions follow by
choosing a finite flat atlas of $\cG$ and computing morphisms as the equalizer of
the two pullback maps on its finite relation; tensoring with the flat
extension $L/k$ preserves this equalizer.

Composition with $q$ is a $k$-linear map
\[
q_*:\Hom_{\cG}(\cE'',\cE)\ra\operatorname{End}_{\cG}(\cE'').
\]
Semisimplicity over $L$ says that
$\operatorname{id}_{\cE''}\otimes1$ lies in the image of $q_*\otimes_kL$.
Therefore the class of $\operatorname{id}_{\cE''}$ in $\operatorname{coker}(q_*)$
vanishes after tensoring with $L$, and faithful flatness implies that
it already vanishes over $k$.  A preimage of $\operatorname{id}_{\cE''}$ is a
splitting of $q$.  Thus every exact sequence splits.  Finally, the
character projectors supplied by
\cite[Expos\'e~I, Proposition~4.7.3]{SGA3I} are exact and give the
asserted semisimple direct summands.
\end{proof}

\begin{lemm}\label{lem:intrinsic-regular-bundle}
Every finite gerbe $\cA/k$ whose geometric inertia groups have order
invertible in $k$ has a vector bundle which is faithful on geometric
inertia.  This construction does not require the band of $\cA$ to be
represented over $k$.
\end{lemm}

\begin{proof}
Since $\cA$ becomes neutral over a separable closure and is of finite
presentation, choose an object $a\in\cA(L)$ over a finite separable
extension $L/k$.
The corresponding morphism
\[
u:\Spec L\ra\cA
\]
is representable, finite locally free, and surjective.  Indeed, after
base change by an object $b\in\cA(T)$, the fiber is the finite flat
$\operatorname{Isom}$-torsor between $b$ and the pullbacks of $a$.
Hence
\[
\mathcal R_a:=u_*\mathcal O_{\Spec L}
\]
is a vector bundle on $\cA$.  Let $\bar k$ be an algebraic closure and
choose a geometric object $\bar a\in\cA(\bar k)$.  Put
$P=\underline{\Aut}_{\bar k}(\bar a)$.  Since $L/k$ is separable,
$\Spec(L\otimes_k\bar k)$ is the disjoint union of
$[L:k]$ copies of $\Spec\bar k$.  If $a_\sigma$ denotes the object
coming from one such component, then
\[
(\Spec\bar k)_{a_\sigma}
\times_{\cA_{\bar k}}
(\Spec\bar k)_{\bar a}
=\underline{\operatorname{Isom}}_{\bar k}(\bar a,a_\sigma)
\]
for that component of the fiber product.  This is a nonempty
$P$-torsor, hence isomorphic over $\bar k$ to $P$ as a
$P$-scheme.  Finite base change therefore gives
\[
\bar a^*\mathcal R_a
\simeq
\H^0\!\left(
\Spec(L\otimes_k\bar k)\times_{\cA_{\bar k}}\Spec\bar k,
\mathcal O
\right)
\simeq
\mathcal O(P)^{\oplus[L:k]}.
\]
Thus the geometric fiber is explicitly a nonzero direct sum of regular
$P$-modules.  The regular action of a finite group scheme on its
coordinate algebra is faithful.
Thus $\mathcal R_a$ is faithful on geometric inertia.
\end{proof}

If $\fC$ is a split central diagonalizable subband, the restriction of a
regular $P$-module to $\fC_{\bar a}$ contains every
character of $\fC$.
Consequently $(\mathcal R_a)_\chi\ne0$ for every $\chi\in\fC^*$;
in particular, every rank ideal introduced above has a positive
generator.

\subsection{The generic central fiber}\label{sec:generic-central-fibre}

Rigidification by $\fC$ gives a morphism
\[
q:\cG\ra\cH:=\cG\!\sslash\fC
\]
with relative inertia $\fC$; see \cite[Section~5.1]{ACV}.  The stack
$\cH$ is again a finite gerbe whose geometric inertia groups have order
invertible in $k$.

For a vector bundle $\mathcal F$ of rank $r$ on $\cH$, write
$\operatorname{Fr}(\mathcal F)$ for the open substack of
\[
\Tot\!\left(
\underline{\operatorname{Hom}}_{\cH}
(\mathcal O_{\cH}^{\oplus r},\mathcal F)
\right)
\]
parametrizing frames of $\mathcal F$, that is, homomorphisms
$\mathcal O_{\cH}^{\oplus r}\to\mathcal F$ which are isomorphisms, where $\Tot$ denotes the total space.

\begin{lemm}\label{lem:generic-object-rigidification}
Let $\mathcal F$ be a faithful vector bundle on $\cH$.  Then
$U:=\operatorname{Fr}(\mathcal F)$ is an integral algebraic space.  If
$K=k(U)$, the generic point of $U$ determines an object
\[
\eta:\Spec K\ra\cH.
\]
\end{lemm}

\begin{proof}
Put $r=\rk(\mathcal F)$ and
\[
\mathcal V:=\underline{\operatorname{Hom}}_{\cH}
(\mathcal O_{\cH}^{\oplus r},\mathcal F)
\simeq \mathcal F^{\oplus r}.
\]
Let $h$ be a geometric point of $\cH$. Any automorphism of $h$ that fixes a frame in $\mathcal F_h$ must act as the identity on the entire fiber $\mathcal F_h$. Because $\mathcal F$ is faithful on geometric inertia, such an automorphism is necessarily trivial. It follows that the inertia of $U$ is trivial, and hence $U$ is an algebraic space.

The morphism $U\to\cH$ is the frame bundle of $\mathcal F$, hence a
$\GL_r$-torsor and in particular smooth and surjective.  

Moreover $U$ is integral. To see this, after base change to an algebraic closure and a choice of object of $\cH$, write $\cH_{\bar k}\simeq BG$ and let
$G\hookrightarrow\GL_r$ be the faithful representation defined by
$\mathcal F$.  Then
\[
U_{\bar k}\simeq \GL_r/G.
\]
The quotient map $\GL_r\to\GL_r/G$ is finite \'{e}tale and surjective.
Hence $\GL_r/G$ is irreducible, and smoothness descends along this cover, so
$\GL_r/G$ is reduced.  Thus $U_{\bar k}$ is integral, and so is $U$.
The generic point of $U$ gives the asserted object $\eta$.
\end{proof}

\begin{rema}
Although we do not use this terminology below, the frame morphism
$U\to\cH$ is versal in the usual sense; compare
\cite[\S 1]{DuncanReichstein2015}.  Indeed, for every nonempty open
$U^\circ\subset U$, every extension $L/k$ with $L$ infinite, and every
$h\in\cH(L)$, the morphism $\Spec L\to\cH$ is faithfully flat and surjective, so the
pullback $U_h^\circ=U^\circ\times_{\cH,h}\Spec L$ is nonempty and open in
$\operatorname{Fr}(\mathcal F_h)\simeq\GL_{r,L}$, hence has an
$L$-rational point.
\end{rema}

We keep the notation $\eta$ for the generic frame object of
Lemma~\ref{lem:generic-object-rigidification}. Since
$q:\cG\to\cH$ is an fppf gerbe banded by $\fC$, its fiber over $\eta$,
\[
\cX_\eta=\cG_K\times_{\cH_K,\eta}\Spec K,
\]
is a gerbe banded by $\fC_K$.

When $\cH=BQ$ is neutral, the frame space is $\GL_r/Q$, where
$Q\hookrightarrow\GL_r$ is the faithful representation defined by
$\mathcal F$.  Its generic point gives the same generic $Q$-torsor used in
\cite[Section~4, before Lemma~4.3]{KM}.  Thus $\eta$ is the intrinsic
substitute for that generic torsor when the rigidification need not be a
classifying stack over $k$.

For $\chi\in\fC^*$, push the class of $\cX_\eta$ along
$\chi:\fC_K\to\mathbb G_{m,K}$:
\[
\beta_\eta(\chi)=\chi_*[\cX_\eta]\in\Br(K).
\]
Pullback to $\cX_{\eta}$ of a $\chi$-bundle on $\cG$ is a
$\beta_\eta(\chi)$-twisted vector space.  The index of a Brauer class divides
the rank of every twisted vector space of that class; see
\cite[Proposition~3.1.2.1]{Lieblich2008}.  Taking the greatest common divisor of the
ranks of all $\chi$-bundles therefore gives
\begin{equation}\label{eq:index-divides-rank-ideal}
\ind\beta_\eta(\chi)\mid d_\chi(\cG).
\end{equation}

\subsection{The \texorpdfstring{$G$}{G}-theoretic comparison}

\begin{lemm}\label{lem:generic-central-g0-surjection}
For every $\chi\in\fC^*$, restriction along $\cX_\eta\to\cG$ induces a
surjection
\[
G_0^\chi(\cG)\twoheadrightarrow G_0^\chi(\cX_\eta).
\]
\end{lemm}

\begin{proof}
Consider the sequence
\[
G_0^\chi(\cG)
\rightarrow
G_0^\chi\bigl(\cG\times_{\cH}\Tot(\cV)\bigr)
\rightarrow
G_0^\chi\bigl(\cG\times_{\cH}U\bigr)
\rightarrow
G_0^\chi(X_\eta).
\]

The first arrow is an isomorphism by homotopy invariance for
$G$-theory \cite[Corollary~3.18]{KhanKG}. Indeed,
\[
\cG\times_{\cH}\Tot(\cV)
\simeq
\Tot(q^*\cV)
\]
is the total space of the vector bundle $q^*\cV$ on $\cG$.
With Khan's convention, $\mathbf V_{\cG}(\mathcal E)$ parametrizes
cosections of $\mathcal E$, so
\[
\Tot(q^*\cV)
\simeq
\mathbf V_{\cG}\bigl((q^*\cV)^\vee\bigr),
\]
and Corollary~3.18 applies to $(q^*\cV)^\vee$.

The second arrow is surjective by localization
\cite[Theorem~3.9 and Proposition~3.1]{KhanKG}, since
\[
\cG\times_{\cH}U
\subset
\cG\times_{\cH}\Tot(\cV)
\]
is an open substack. Here Khan's $G$-theory agrees with the
$K$-theory of coherent sheaves used in this article by
\cite[Proposition~3.3]{KhanKG}. Moreover, since the
$\fC$-character decomposition is functorial and the $\chi$-projection
is exact, both homotopy invariance and localization preserve the
$\chi$-summands. Thus the preceding statements give respectively
an isomorphism and a surjection on $G_0^\chi$.

It remains to prove that
\[
G_0^\chi\bigl(\cG\times_{\cH}U\bigr)
\rightarrow
G_0^\chi(X_\eta)
\]
is surjective. Let $\mathcal E$ be a coherent sheaf of pure
$\fC$-character $\chi$ on $X_\eta$. Since $\eta$ is the generic
point of the integral algebraic space $U$, the sheaf $\mathcal E$
spreads out, after shrinking $U$, to a coherent sheaf on
\[
\cG\times_{\cH}W
\]
for some nonempty open subspace $W\subset U$. Applying the exact
$\chi$-projection, we may assume that this extension is of pure
character $\chi$. Localization for the open immersion
\[
\cG\times_{\cH}W
\subset
\cG\times_{\cH}U
\]
then shows that its class lifts to
\[
G_0^\chi\bigl(\cG\times_{\cH}U\bigr).
\]
Hence the last arrow is surjective, and therefore so is the
composite.
\end{proof}

So the above lemma is a stacky version of the $K$-theory argument in the proof of \cite[Theorem~4.4]{KM}, without assuming that the gerbe $\cG$ has a group scheme defined over $k$ as its band. 

\begin{theo}[Non-neutral index theorem]
\label{thm:nonneutral-central-index}
In the notation above, for every $\chi\in\fC^*$,
\begin{equation}\label{eq:rank-ideal-index}
I_\chi(\cG)=\ind\beta_\eta(\chi)\mathbb Z.
\end{equation}
In particular,
\[
d_\chi(\cG)=\ind\beta_\eta(\chi).
\]
\end{theo}

\begin{proof}
By Lemma~\ref{lem:generic-central-g0-surjection}, restriction gives a
surjection
\[
G_0^\chi(\cG)\twoheadrightarrow G_0^\chi(\cX_\eta).
\]
The target is the Grothendieck group of
$\beta_\eta(\chi)$-twisted $K$-vector spaces, whose rank image is
$\ind\beta_\eta(\chi)\mathbb Z$.  Hence this ideal is contained in
$I_\chi(\cG)$.  The reverse containment is
\eqref{eq:index-divides-rank-ideal}, proving
\eqref{eq:rank-ideal-index}.
\end{proof}

\begin{coro}
\label{cor:independence-central-index}
The numbers $\ind\beta_\eta(\chi)$ are independent of all choices in
Lemma~\ref{lem:generic-object-rigidification}.  Moreover, the single generic central fiber $\cX_\eta$ realizes the
rank ideals for every $\chi\in\fC^*$ simultaneously.
\end{coro}

\begin{proof}
Theorem~\ref{thm:nonneutral-central-index} identifies every index with the
intrinsic integer $d_\chi(\cG)$.
\end{proof}

\section{The gerby Karpenko-Merkurjev theorem under \texorpdfstring{$p$}{p}-primary descent}
\label{sec:gerby-km}

Throughout this section, $p$ is a prime, $\operatorname{char}(k)\ne p$,
and $\cG/k$ is a finite gerbe whose geometric inertia groups are finite
$p$-groups.  We write
$\band=\operatorname{Band}(\cG)$ in the sense of
Definition~\ref{def:band}.  It is a finite $p$-band:
after a finite separable extension it is represented by a constant finite
$p$-group.  We do not assume that $\cG$ is neutral or that $\band$
is represented by a group scheme over $k$.

The strategy combines the Karpenko--Merkurjev lower bound argument with
the split central subgroup method for twisted algebraic groups. More
precisely, the $(\mu_p)^s$-banded or neutral gerbe calculation comes from
\cite[Theorems~2.1 and~3.1]{KM}, while the central subband method used in this article has as its origin in
\cite[Definition~4.2]{LMMRTori}; the normal-subgroup criterion parallels
\cite[Lemma~2.1, Theorem~7.1]{LMMRTori}. The new issue is
to carry out these constructions for a band, rather than a global group
scheme, and to combine them with the non-neutral index theorem of
Section~\ref{sec:central-index}.

\subsection{The split central subband}

We say that $\band$ has \emph{$p$-primary descent} if some finite
Galois extension of $p$-power degree represents it by a constant
$p$-group.  Constant bands have $p$-primary descent.  So do many
nonconstant forms; the condition concerns the outer descent action, not
the class of the gerbe.

Recall that a field $k$ is \emph{$p$-special} if every finite extension of
$k$ has $p$-power degree.  In that case $k$ contains a primitive $p$th root
of unity: adjoining one has degree dividing $p-1$.  Moreover, every finite
$p$-band over a $p$-special field has $p$-primary descent.

By Proposition~\ref{prop:center-of-band}, the center $Z(\band)$ is a finite commutative group scheme over $k$.  For a finite
commutative group scheme $\fA$ killed by $p$, let $\Split_k(\fA)$ denote its
maximal split subgroup scheme.  Here a subgroup scheme is \emph{split} if it
is $k$-isomorphic to $(\mu_p)^r$ for some $r\ge0$.  Assume that
$\mu_p\subset k$, and define
\[
\fC(\band):=\Split_k(Z(\band)[p]).
\]
It is a central subband and
\[
\fC(\band)\simeq(\mu_p)^c
\]
for some $c\ge0$.

The following proposition is the band analogue of
\cite[Lemma~2.1, Theorem~7.1]{LMMRTori}.

\begin{prop}\label{prop:split-central-subband}
Assume that $\mu_p\subset k$ and that $\band$ has $p$-primary
descent.  If $\band'$ is a normal subband of $\band$ and
\[
\band'\cap\fC(\band)=1,
\]
then $\band'=1$.
\end{prop}

\begin{proof}
Choose a finite Galois extension $L/k$ of $p$-power degree which
represents and splits $\band$, and write $P$ for the resulting finite
constant $p$-group.  The Galois group
$\Gamma=\operatorname{Gal}(L/k)$ is a $p$-group and acts on $Z(P)[p]$. Under this action, $\fC(\band)_L$ corresponds to the constant subgroup
\[
Z(P)[p]^\Gamma\subset Z(P)[p].
\]

If $\band'\ne1$, it is represented over $L$ by a nontrivial normal
subgroup $N\triangleleft P$.  The assertion that $\band'$ descends as
a normal subband says exactly that $N$ is stable under the outer descent
action; this is well defined because inner automorphisms preserve every
normal subgroup.
The elementary theory of finite $p$-groups gives
\[
W:=N\cap Z(P)[p]\ne0.
\]
The group $W$ is an $\mathbb F_p$-vector space stable under
$\Gamma$.  Every nonzero $\mathbb F_p$-representation of a finite
$p$-group has a nonzero fixed vector: equivalently, orbit counting gives
$|W^\Gamma|\equiv |W|\equiv0\pmod p$.  Hence
\[
1\ne W^\Gamma
\subset N\cap Z(P)[p]^\Gamma
=N\cap\fC(\band)_L.
\]
Thus $\band'\cap\fC(\band)\ne1$, this completes the proof.
\end{proof}

Because $\fC=\fC(\band)$ is split, every vector bundle $\cE$ on $\cG$
has a canonical central-character decomposition
\begin{equation}\label{eq:p-central-decomposition}
\cE=\bigoplus_{\chi\in\fC^*}\cE_\chi,
\quad
\fC^*=\Hom(\fC,\bG_m).
\end{equation}
Here $\fC$ acts on $\cE_\chi$ through the scalar character $\chi$.  Let
\[
\Vect_\chi(\cG)\subset\Vect(\cG)
\]
be the full subcategory of pure central character $\chi$, and put
\begin{equation}\label{eq:block-minimum}
m_\chi(\cG)=
\min\{\rk(\cE):0\ne \cE\in\Vect_\chi(\cG)\}.
\end{equation}

\begin{prop}\label{prop:central-character-main}
If $\mu_p\subset k$ and $\band$ has $p$-primary descent, then
\begin{equation}\label{eq:central-character-formula-main}
\rdim(\cG)=
\min_{\mathcal B\text{ basis of }\fC^*}
\sum_{\chi\in\mathcal B}m_\chi(\cG).
\end{equation}
\end{prop}

\begin{proof}
Let $\cE$ be a vector bundle.  By Lemma~\ref{lem:kernel-subband}, the
kernel of its inertia representation determines a normal subband
$\band_{\cE}\subset\band$, and the restriction of $\cE$ to $\fC$ is
faithful if and only if $\band_{\cE}\cap\fC=1$.
Proposition~\ref{prop:split-central-subband} therefore shows that $\cE$
is faithful on geometric inertia if and only if it is faithful on $\fC$.

By \eqref{eq:p-central-decomposition}, the restriction to $\fC$ is faithful
exactly when the occurring characters span the $\mathbb F_p$-vector space
$\fC^*$.  From any spanning collection one can retain a basis and discard
the other blocks without losing faithfulness.  Each retained block has rank at
least its minimum \eqref{eq:block-minimum}; conversely, minimum blocks for a
basis have faithful direct sum.  Taking the minimum over all bases proves
\eqref{eq:central-character-formula-main}.
\end{proof}

\subsection{Attainment of central rank ideals}

The rank ideal in Definition~\ref{def:rank-ideal} need not be
generated by the rank of an object over a general field.  The arithmetic of
$p$-groups with $p$-primary descent removes this obstruction.

For a constant finite group $P$, we say that a field extension $M/k$ is a
\emph{splitting field for $P$} if every irreducible $M$-representation of
$P$ is absolutely irreducible; equivalently, since
$\operatorname{char}(M)\nmid |P|$, the semisimple algebra $M[P]$ is a
product of full matrix algebras over
$M$.  This representation-theoretic use of ``splitting'' is distinct from
the earlier notion of a split group scheme or split subgroup scheme.

\begin{lemm}\label{lem:p-power-neutralization}
Assume that $\mu_p\subset k$ and that $\band$ has $p$-primary
descent.  There is a finite separable extension $M/k$ of $p$-power
degree with the following properties:
\begin{enumerate}[label=\textup{(\roman*)}]
\item the band $\band_M$ is represented by a constant $p$-group $P$;
\item the gerbe $\cG_M$ is neutral; and
\item $M$ is a splitting field for $P$.
\end{enumerate}
\end{lemm}

\begin{proof}
Choose a finite Galois extension $L/k$ of $p$-power degree over which
the band is the constant group $P$.  Let $p^e$ be the exponent of
$P$, and put
\[
L_1=L(\mu_{p^e}).
\]
Since $\mu_p\subset k$, the extension $L_1/L$ has $p$-power degree.
For odd $p$, this follows because the subgroup of
$(\mathbb Z/p^e\mathbb Z)^\times$ acting trivially on $\mu_p$ is a
$p$-group; for $p=2$, the whole group
$(\mathbb Z/2^e\mathbb Z)^\times$ is a $2$-group.  Thus
$[L_1:k]$ is a power of $p$.

Fix a banding of $\cG_{L_1}$ by $P$.  The neutral gerbe $BP$, with
its standard banding, supplies a base point in the set of equivalence
classes of gerbes with this fixed band.  Giraud's central twisting theorem
says that $\H^2(L_1,Z(P))$ acts simply transitively on this pointed set
\cite[Chapter~IV, Theorem~3.3.3(i)]{Giraud}.  Consequently there is a
unique
\[
\alpha\in \H^2(L_1,Z(P))
\]
which carries $BP$ to $\cG_{L_1}$.

The finite abelian group $Z(P)$ has exponent dividing $p^e$.  Because
$\mu_{p^e}\subset L_1$, after choosing a decomposition of $Z(P)$ into
cyclic factors we may identify
\[
Z(P)\simeq\prod_{j=1}^r\mu_{p^{a_j}},
\quad a_j\le e.
\]
The Kummer sequence then identifies $\alpha$ with a tuple
\[
(\alpha_1,\ldots,\alpha_r),
\quad
\alpha_j\in\Br(L_1)[p^{a_j}].
\]
The index of every $\alpha_j$ is a power of $p$, since period and
index have the same prime divisors
\cite[Chapter~4]{GilleSzamuely}.  Choose a maximal separable subfield
of a division algebra representing $\alpha_j$; it is a splitting field
of degree $\ind(\alpha_j)$.  The compositum $M$ of these splitting
fields has $p$-power degree over $L_1$.  The class $\alpha_M$ is
zero, so $\cG_M\simeq BP_M$.  Finally, a field of characteristic prime
to $|P|$ containing a primitive root of unity of order equal to the
exponent of $P$ is a splitting field for $P$
\cite[Chapter~12]{SerreRep}.  This proves all three assertions.
\end{proof}

Recall that a Hom-finite semisimple $M$-linear category is
\emph{split semisimple} if the endomorphism algebra of every simple object
is $M$.

\begin{lemm}\label{lem:semisimple-p-descent}
Let $\cA$ be a Hom-finite semisimple $k$-linear category equipped with an
integer valued rank preserved by scalar extension.  Suppose that, for a
finite separable extension $M/k$ of $p$-power degree, the category
$\cA_M$ is split semisimple and all of its simple objects have
$p$-power rank.  Then every simple object of $\cA$ has $p$-power
rank.
\end{lemm}

\begin{proof}
Let $S$ be simple.  Write
$D=\operatorname{End}_{\cA}(S)$ and let $F=Z(D)$.
Then $D$ is a finite dimensional division algebra, central over the
finite extension $F/k$.  Since $\cA_M$ is split and $M/k$ is
separable, $F/k$ is separable,
\[
F\otimes_kM\simeq M^{[F:k]},
\]
and every factor of $D\otimes_kM$ is a matrix algebra over $M$.
After choosing one factor, we may regard $F$ as a subfield of $M$,
and $M/F$ splits $D$.  Hence
\[
[F:k]\mid[M:k],
\quad
\deg_F(D)=\ind(D)\mid[M:F].
\]
Both integers are powers of $p$.

Over a separable closure $k_s$, Artin--Wedderburn theory gives
\begin{equation}\label{eq:aw-simple-descent}
S_{k_s}\simeq
\bigoplus_{V\in\mathcal O}V^{\oplus s},
\quad
|\mathcal O|=[F:k],
\quad
s=\deg_F(D).
\end{equation}
Here $\mathcal O$ is one Galois orbit of absolutely simple objects.
Thus $|\mathcal O|=[F:k]$ and $s=\ind(D)$ are powers of $p$.
The objects in $\mathcal O$ have a common rank, say $p^a$, because
they are Galois conjugate and the simple ranks after extension to $M$
are powers of $p$.
Taking ranks in \eqref{eq:aw-simple-descent} gives
\[
\rk(S)=|\mathcal O|\,s\,p^a,
\]
which is a power of $p$.
\end{proof}

\begin{rema}\label{rem:semisimple-descent-lmmr}
The preceding lemma abstracts \cite[Lemma~5.2]{LMMRTori}, which proves the
corresponding statement for irreducible representations of a finite
$p$-group over a $p$-closed field.
\end{rema}

\begin{lemm}\label{lem:p-power-ranks-main}
Assume that $\mu_p\subset k$ and that $\band$ has $p$-primary
descent.  Every simple object of
$\Vect(\cG)$ has rank a power of $p$.  Consequently,
\begin{equation}\label{eq:block-gcd}
m_\chi(\cG)=d_\chi(\cG).
\end{equation}
\end{lemm}

\begin{proof}
By Lemma~\ref{lem:p-power-neutralization}, there is a finite separable
extension $M/k$ of $p$-power degree such that
\[
\cG_M\simeq BP_M
\]
and $M$ is a splitting field for $P$.  Hence
$\Vect(\cG_M)\simeq\operatorname{Rep}_M(P)$ is split semisimple.  The ranks of its
simple objects are the ranks of the absolutely irreducible
representations of the finite $p$-group $P$, and these ranks are
powers of $p$.  Lemma~\ref{lem:semisimple-p-descent}, applied to
$\Vect(\cG)$, now shows that every simple object over $k$ has
$p$-power rank.

By Lemma~\ref{lem:semisimple-gerbe}, the ranks in
$\Vect_\chi(\cG)$ are nonnegative integral combinations of the ranks of its
simple objects.  Their greatest common divisor is therefore the greatest
common divisor of a nonempty collection of powers of $p$, which is its
smallest member.  The smallest simple rank is also the smallest rank of any
nonzero object, proving \eqref{eq:block-gcd}.
\end{proof}

\begin{rema}
The conclusion of Lemma~\ref{lem:p-power-ranks-main} can fail before
passing to a $p$-closure if the band does not have $p$-primary
descent.  Galois orbits of irreducibles can then have prime-to-$p$
cardinality, and representation dimension can drop after a prime-to-$p$ extension.
\end{rema}

\subsection{Essential-dimension inequalities}

The following result is known; see \cite[Theorem~3.1]{KM}. We repackage it with our notions.

\begin{theo}
\label{thm:km-elementary-computation}
Let $F$ be a field of characteristic different from $p$, and let $\cX/F$
be a gerbe banded by $\fC=(\mu_p)^c$.  If
\[
\beta:\fC^*\ra\Br(F)[p]
\]
is its character-obstruction map, then
\begin{equation}\label{eq:km-elementary-basis}
\ed_F(\cX;p)=\ed_F(\cX)=
\min_{\mathcal B\text{ basis of }\fC^*}
\sum_{\chi\in\mathcal B}\ind\beta(\chi).
\end{equation}
\end{theo}

\begin{proof}
Karpenko--Merkurjev prove
\[
\ed_F(\cX;p)=\ed_F(\cX)=\cdim_p(\im\beta)+c
\]
for a $(\mu_p)^c$-gerbe, where $\cdim_p$ denotes the $p$-canonical dimension, and compute the canonical $p$-dimension of the Brauer subgroup by a minimum-index basis \cite[Theorems~2.1 and~3.1 and Remark~2.9]{KM}. Choose lifts in $\fC^*$
of such a basis of $\im\beta$, then complete them by a basis of $\ker\beta$. Characters in the kernel contribute index $1$, so this gives the right-hand side of \eqref{eq:km-elementary-basis}. Conversely, the images of any basis of $\fC^*$ generate $\im\beta$, and the minimum basis statement gives the reverse inequality.
\end{proof}

\begin{lemm}\label{lem:finite-fiber-ed}
Let $F$ be a field, let $f:\cA\to\cB$ be a morphism of finite gerbes, and
let $\eta\in\cB(F)$.  Put
\[
\cZ=\cA\times_{\cB,\eta}\Spec F.
\]
Then
\[
\ed_F(\cA;p)\ge\ed_F(\cZ;p).
\]
\end{lemm}

Compare \cite[Theorem~4.2]{KM}, where the same proof is
used for the gerbe attached to a central quotient of a finite $p$-group.

\begin{proof}
An object of $\cZ(L)$ is a pair $(x,\theta)$ with $x\in\cA(L)$ and
\[
\theta:f(x)\xrightarrow{\sim}\eta_L.
\]
Let $L'/L$ be a finite extension of degree prime to $p$, and suppose that
$x_{L'}$ descends to $x_0\in\cA(L_0)$ for an intermediate field
$F\subset L_0\subset L'$.  The isomorphism $\theta_{L'}$ is an $L'$-point
of the finite $L_0$-scheme
\[
\operatorname{Isom}_{\cB}\bigl(f(x_0),\eta_{L_0}\bigr).
\]
Its scheme-theoretic image has residue field $L_1$ finite over $L_0$ and
contained in $L'$.  Over $L_1$ the isomorphism descends, so the pair
$(x,\theta)_{L'}$ descends to $L_1$.  Since
\[
\trdeg_F(L_1)=\trdeg_F(L_0),
\]
adding the isomorphism does not increase the number of parameters.  Taking
the infimum over prime-to-$p$ extensions and descent fields, and then the
supremum over objects of $\cZ$, proves the inequality.
\end{proof}

\begin{lemm}\label{lem:ed-base-change}
For every field extension $K/k$ and every category fibered in groupoids
$\cA/k$,
\[
\ed_k(\cA;p)\ge\ed_K(\cA_K;p).
\]
\end{lemm}

\begin{proof}
If an object over a $K$-field descends, after a prime-to-$p$ extension, to a
$k$-subfield $L_0$, then after adjoining $K$ it descends to $KL_0$.
Submodularity of transcendence degree gives
\[
\trdeg_K(KL_0)\le\trdeg_k(L_0).
\]
The degree of the corresponding compositum extension still divides a
prime-to-$p$ degree.  Taking minima and suprema gives the claim.
\end{proof}

\begin{theo}
\label{thm:central-rank-lower-bound}
Let $\cA/k$ be a finite gerbe whose geometric inertia groups have order
invertible in $k$, let $\fC\simeq(\mu_p)^c$ be a split central subband,
and assume $\operatorname{char}(k)\ne p$.  Then
\begin{equation}\label{eq:central-rank-lower-bound}
\ed_k(\cA;p)\ge
\min_{\mathcal B\text{ basis of }\fC^*}
\sum_{\chi\in\mathcal B}d_\chi(\cA).
\end{equation}
No $p$-group hypothesis is imposed on the full inertia.
\end{theo}

\begin{proof}
Rigidify $\cA$ by $\fC$, and let
$\eta\in(\cA\sslash\fC)(K)$ be the generic frame object of
Lemma~\ref{lem:generic-object-rigidification}.  Its central fiber
$\cX_\eta$ is a $\fC_K$-gerbe.  Lemma~\ref{lem:ed-base-change} and
Lemma~\ref{lem:finite-fiber-ed} give
\[
\ed_k(\cA;p)\ge
\ed_K(\cA_K;p)\ge
\ed_K(\cX_\eta;p).
\]
Theorem~\ref{eq:km-elementary-basis} computes the last term as
\[
\min_{\mathcal B\text{ basis of }\fC^*}
\sum_{\chi\in\mathcal B}\ind\beta_\eta(\chi).
\]
Then theorem~\ref{thm:nonneutral-central-index} identifies each
$\ind\beta_\eta(\chi)$ with $d_\chi(\cA)$, proving
\eqref{eq:central-rank-lower-bound}.
\end{proof}

\begin{prop}\label{prop:ed-upper-main}
Let $\cG/k$ be a finite gerbe whose geometric inertia groups have order
invertible in $k$.  Then
\[
\ed_k(\cG;p)\le\ed_k(\cG)\le\rdim(\cG).
\]
\end{prop}

\begin{proof}
Let $\cE$ be faithful of rank $r$, and let $U\subset\Tot(\cE)$ be the
locus of trivial inertia.  We first check that $U$ is nonempty.  After a
finite separable extension $k'/k$ which neutralizes $\cG$ and makes its inertia
constant, write $\cG_{k'}\simeq BP$ and let $V$ be the representation
corresponding to $\cE$. The group $P$ is finite and constant, and $V$ is faithful. Therefore
\[
V^{\mathrm{free}}
=V\setminus\bigcup_{1\ne g\in P}V^g
\]
is a nonempty $P$-invariant open subset, since each fixed-point space $V^g$ is a proper linear subspace. Consequently, after this extension the free inertia locus in $\Tot(\cE)$ is nonempty, and thus it was already nonempty over $k$. On this locus the inertia is trivial, so it is an algebraic space, and $\dim U=r$.

Take $x\in\cG(L)$.  If $x$ descends to an algebraic extension of $k$,
finite presentation descends it further to a finite extension, so its
essential dimension is zero.  Otherwise finite presentation lets us
descend $x$ to a finitely generated field extension of $k$ and replace
$L$ by that field.  This field has positive transcendence degree over $k$,
and in particular is infinite.  The fiber $U_x$ is a nonempty open subset
of the affine space $\cE_x$, so $U_x(L)\ne\varnothing$.  Such a point is a
lift $u\in U(L)$ of $x$.
The morphism $\Spec L\to U$ factors through the residue field of its
image, whose transcendence degree over $k$ is at most
$\dim U=r$.  Forgetting the vector in the pair $u=(x,v)$ shows that
$x$ descends to the same field.  Thus $\ed_k(\cG)\le r$.

Minimizing over $\cE$ proves the second inequality; the first is formal
from the definition of essential dimension at a prime.
\end{proof}

\begin{theo}
\label{thm:gerby-km-special}
Assume that $\mu_p\subset k$ and that the band $\band$ has
$p$-primary descent.  Then
\begin{equation}\label{eq:gerby-km-special}
\ed_k(\cG)=\ed_k(\cG;p)=\rdim(\cG).
\end{equation}
\end{theo}

\begin{proof}
Use $\fC=\fC(\band)$ in Theorem~\ref{thm:central-rank-lower-bound}. Lemma~\ref{lem:p-power-ranks-main} identifies each
$d_\chi(\cG)$ with $m_\chi(\cG)$, and
Proposition~\ref{prop:central-character-main} identifies the resulting
minimum with $\rdim(\cG)$.  Thus
\[
\ed_k(\cG;p)\ge
\min_{\mathcal B\text{ basis of }\fC^*}
\sum_{\chi\in\mathcal B}d_\chi(\cG)
=
\min_{\mathcal B}
\sum_{\chi\in\mathcal B}m_\chi(\cG).
\]
The last expression is $\rdim(\cG)$.  Hence
\[
\rdim(\cG)\le\ed_k(\cG;p).
\]
Together with Proposition~\ref{prop:ed-upper-main}, we obtain
\[
\rdim(\cG)
\le\ed_k(\cG;p)
\le\ed_k(\cG)
\le\rdim(\cG),
\]
which proves \eqref{eq:gerby-km-special}.
\end{proof}

\begin{coro}
\label{cor:constant-band}
Let $k$ contain a primitive $p$th root of unity, and let $\cG/k$ be
a finite gerbe with constant finite $p$-group band $P$.  Then, without
assuming that $\cG$ is neutral,
\[
\ed_k(\cG)=\ed_k(\cG;p)=\rdim(\cG).
\]
\end{coro}

\begin{proof}
A constant band has $p$-primary descent, with splitting extension $k/k$.
Apply Theorem~\ref{thm:gerby-km-special}.
\end{proof}

\section{Prime-to-\texorpdfstring{$p$}{p} localization and residual gerbes}
\label{sec:p-localization}

In this section, we pass to a $p$-closure of $k$ and proves
the general Karpenko--Merkurjev theorem.

Fix an algebraic closure $\bar k$, a separable closure
$k_s\subset\bar k$, and a Sylow pro-$p$ subgroup
$\Phi\subset\operatorname{Gal}(k_s/k)$.  Following
\cite[Section~2]{LMMR13}, define
\[
k^{(p)}
=
\{a\in\bar k:
a\text{ is purely inseparable over }k_s^\Phi\}.
\]
This is a $p$-closure of $k$, and it is $p$-special: it is perfect,
it is the filtered
union of finite extensions of $k$ of degree prime to $p$, and every
finite extension of $k^{(p)}$ has $p$-power degree
\cite[Lemma~2.1, Lemma~2.2]{LMMR13}.  Define
\[
\rdim_p(\cG):=\rdim(\cG_{k^{(p)}}).
\]

\begin{lemm}\label{lem:p-local-rdim}
The number $\rdim_p(\cG)$ is independent of the chosen $p$-closure
and satisfies
\[
\rdim_p(\cG)=
\min_{\substack{L/k\text{ finite}\\p\nmid[L:k]}}
\rdim(\cG_L).
\]
\end{lemm}

\begin{proof}
Let $\cE$ be a minimum rank faithful bundle over $\cG_{k^{(p)}}$. The bundle $\cE$, together with its inertia action, is of finite
presentation, so it descends to a vector bundle $\cE_L$ on $\cG_L$
for some finite subextension $L/k$ of $k^{(p)}/k$; necessarily
$p\nmid[L:k]$.  The descended representation is already faithful over
$L$: its kernel subband with repect to $\cE$ is trivial after base change to $k^{(p)}$, so it descends to the trivial group over $L$. Hence $\cE_L$ is faithful of the same rank as $\cE$.  This proves
\[
\rdim(\cG_{k^{(p)}})\ge
\min_{p\nmid[L:k]}\rdim(\cG_L).
\]

Conversely, let $L/k$ be finite of degree prime to $p$, and first
let $L_0/k$ be its maximal separable subextension.  The subgroup
$\operatorname{Gal}(k_s/L_0)$ has index $[L_0:k]$, which is prime to $p$, and
therefore contains a Sylow pro-$p$ subgroup of
$\operatorname{Gal}(k_s/k)$.
Sylow conjugacy allows us to choose the $k$-embedding
$L_0\hookrightarrow k_s$ so that
$\Phi\subset\operatorname{Gal}(k_s/L_0)$.  Thus
$L_0\subset k_s^\Phi$.  Since
$L/L_0$ is purely inseparable, the definition of $k^{(p)}$ then gives
an embedding $L\hookrightarrow k^{(p)}$.  Consequently
\[
\rdim(\cG_{k^{(p)}})\le\rdim(\cG_L).
\]
Taking the minimum over $L$ gives the reverse inequality.  The displayed
formula does not involve the choice of $\Phi$, so it also proves
independence of the $p$-closure.
\end{proof}

\begin{lemm}\label{lem:p-special-ed-invariance}
Let $\cA$ be a category fibered in groupoids over $k$ whose functor of
isomorphism classes commutes with filtered unions of fields; this holds, in
particular, for an algebraic stack locally of finite presentation.  If
$k^{(p)}$ is a $p$-closure, then
\[
\ed_k(\cA;p)=
\ed_{k^{(p)}}(\cA_{k^{(p)}};p).
\]
\end{lemm}

\begin{proof}
Apply the standard $p$-closure lemma to the limit-preserving functor
\[
K\longmapsto\{\text{isomorphism classes in }\cA(K)\};
\]
see \cite[Lemma~2.3(b)]{LMMR13}.
\end{proof}

\begin{theo}\label{thm:gerby-km}
Let $\operatorname{char}(k)\ne p$, and let $\cG/k$ be a finite gerbe
whose geometric inertia groups are $p$-groups.  Then
\[
\ed_k(\cG;p)=\rdim_p(\cG).
\]
\end{theo}

\begin{proof}
Apply Lemma~\ref{lem:p-special-ed-invariance}, followed by
Theorem~\ref{thm:gerby-km-special} over $k^{(p)}$, to obtain
\[
\ed_k(\cG;p)
=\ed_{k^{(p)}}(\cG_{k^{(p)}};p)
=\rdim(\cG_{k^{(p)}})
=\rdim_p(\cG).
\]
\end{proof}

\begin{rema}[Why localization is necessary]\label{rem:localization-necessary}
Over an arbitrary field, ordinary representation dimension cannot replace
$\rdim_p$.  Let $p=3$, assume $\mu_3\subset k$, and twist $C_3$
by a nontrivial quadratic extension $L/k$ through
\[
\operatorname{Gal}(L/k)\ra\operatorname{Aut}(C_3)\simeq C_2.
\]
Let $\fG$ be the resulting nonconstant $k$-group scheme.  No Galois
extension of $3$-power degree splits its band: such an extension has
odd degree and cannot contain $L$.
The two nontrivial geometric characters form one Galois orbit, so the
representation dimension over $k$ is $2$, whereas the quadratic extension has
degree prime to three and splits the group; hence
\[
\ed_k(B\fG;3)=\rdim_3(B\fG)=1.
\]
The two conjugate characters give rise to a faithful $2$-dimensional $k$-representation $V$ of $\fG$. The resulting action on $\mathbb P(V)\simeq\mathbb P^1$ is faithful and hence is generically free. Using an alternative description of essential dimension \cite[Definition~3.5]{ReichsteinED}; see also \cite[Section~3d, Proposition~3.13]{MerkurjevSurvey}, it follows that
\[
\ed_k(B\fG)\le \dim\mathbb P(V)=1.
\]
Consequently,
\[
\ed_k(B\fG)=\ed_k(B\fG;3)=\rdim_3(B\fG)=1
<2=\rdim(B\fG).
\]
\end{rema}

\section{Relative faithful rank and lifting problems}
\label{sec:relative-rank}

L\"otscher's fiber-dimension theorem gives, for a morphism
$\cX\to\cY$ of categories fibered in groupoids,
\begin{equation}\label{eq:lotscher-relative}
\ed_k(\cX;p)\leq \ed_k(\cY;p)+
\sup_{K/k,\,y\in\cY(K)}\ed_K(\cX_y;p)
\end{equation}
\cite[Theorem~1.1]{LotscherFiber}.  For a surjection of algebraic groups,
the fibers in this formula are the corresponding gerbes classifying liftings of torsors \cite[Example~3.4]{LotscherFiber}. Thus the supremum itself is classical;
the purpose of this section is to compute it for locally full morphisms of
finite $p$-gerbes.

For a normal subgroup $H\triangleleft G$ of a finite $p$-group, the
locally full morphism relevant below is
\[
BG\to B(G/H),
\]
induced by the quotient homomorphism $G\twoheadrightarrow G/H$.  We show
that the fiber over the generic $G/H$-torsor arising from a free open
subset of a representation realizes the corresponding fiber supremum.  L\"otscher proved this when $H$ is
central and isomorphic to a power of $\mu_p$, and explicitly left the
general normal-subgroup case open \cite[Remark~4.4]{LotscherFiber}.

\subsection{The relative faithful rank}

Let $f:\cX\to\cY$ be a morphism of finite gerbes.  Recall that
\emph{locally full} means that morphisms between objects in the image of
$f$ lift locally; see \cite[Definition~3.4]{BorneVistoliFundamental}.
For finite gerbes over a field, this is equivalent to surjectivity of the
induced homomorphism on geometric inertia groups.
Write
\[
I_{\cX/\cY}=\ker(I_{\cX}\ra f^*I_{\cY})
\]
for its relative inertia.  A vector bundle $\cE$ on $\cX$ is
\emph{$f$-faithful} if the homomorphism
\[
I_{\cX/\cY}\ra\GL(\cE)
\]
is a monomorphism.

\begin{defi}[Relative faithful rank]\label{def:relative-rank}
Set
\[
\rdim(f)=\min\{\rk(\cE):\cE\in\Vect(\cX)
                         \text{ is $f$-faithful}\}.
\]
We put $\rdim(f)=0$ when the relative inertia is trivial.  If
$k^{(p)}$ is a $p$-closure of $k$, set
\[
\rdim_p(f)=\rdim(f_{k^{(p)}}).
\]
\end{defi}

The last number is independent of the chosen $p$-closure and satisfies
\begin{equation}\label{eq:relative-p-local-minimum}
\rdim_p(f)=
\min_{\substack{L/k\text{ finite}\\p\nmid[L:k]}}\rdim(f_L).
\end{equation}
Indeed, the proof of Lemma~\ref{lem:p-local-rdim} applies verbatim to the
relative inertia action.

If $f$ is locally full, every fiber
$\cX_y=\cX\times_{\cY}\Spec K$ is a finite gerbe; this is the
surjective case of \cite[Lemma~3.3]{BorneVistoliFundamental}.  We define
\begin{equation}\label{eq:relative-ed-definition}
\ed_k(f;p)=
\sup_{\substack{K/k\\y\in\cY(K)}}\ed_K(\cX_y;p).
\end{equation}
For the structural morphism $\cX\to\Spec k$, Definitions
\ref{def:relative-rank} and \eqref{eq:relative-ed-definition} recover
$\rdim_p(\cX)$ and $\ed_k(\cX;p)$.

Relative rank does not increase under base change.  It is also
subadditive under composition: if
$\cX\xrightarrow{f}\cY\xrightarrow{g}\cZ$ are locally full, then
\[
\rdim(gf)\leq\rdim(f)+\rdim(g).
\]
To see this, take bundles $\cE_f,\cE_g$ realizing the two terms on the
right.  An element of $I_{\cX/\cZ}$ acting trivially on
$\cE_f\oplus f^*\cE_g$ maps trivially to $I_{\cY/\cZ}$, and hence
lies in $I_{\cX/\cY}$, where $f$-faithfulness of $\cE_f$ forces it to be trivial.  The same assertions
hold after $p$-localization.

\subsection{The relative central subband}

Let $\band_f$ be the normal subband of $\band_{\cX}$ corresponds to $I_{\cX/\cY}$. This is well defined: after choosing
local objects of $\cX$, the kernels of the induced homomorphisms on
inertia are carried to one another by conjugation under change of object,
so Lemma~\ref{lem:subband-descent} descends them to a normal subband,
independently of the choices.  Assume temporarily that $\mu_p\subset k$
and that $\band_{\cX}$ has $p$-primary descent.
Define the \emph{relative split central subband}
\begin{equation}\label{eq:relative-central-subband}
\fC_f=
\Split_k\bigl((\band_f\cap Z(\band_{\cX}))[p]\bigr).
\end{equation}
Thus $\fC_f\simeq(\mu_p)^c$ for some $c\geq0$.  For
$\chi\in\fC_f^*$, we use the notation $d_\chi(\cX)$ from
Definition~\ref{def:rank-ideal} and $m_\chi(\cX)$ from
\eqref{eq:block-minimum}, with the character decomposition taken with
respect to the action of $\fC_f$.

\begin{prop}
\label{prop:relative-faithfulness-criterion}
A vector bundle on $\cX$ is $f$-faithful if and only if its
restriction to $\fC_f$ is faithful.  Consequently
\begin{equation}\label{eq:relative-character-formula}
\rdim(f)=
\min_{\mathcal B\text{ basis of }\fC_f^*}
\sum_{\chi\in\mathcal B}m_\chi(\cX).
\end{equation}
\end{prop}

\begin{proof}
Choose a Galois extension of $p$-power degree which represents the two
bands by a finite $p$-group $P$ and a normal subgroup
$N\triangleleft P$.  If $R\triangleleft P$ is the kernel of
the inertia action on a vector bundle and $R\cap N\ne1$, then
\[
(R\cap N)\cap Z(P)[p]\ne1.
\]
This nonzero $\mathbb F_p$-space is stable under the $p$-group of
descent automorphisms and therefore has a nonzero fixed vector.  Thus
$R\cap N$ meets $\fC_f$ nontrivially.  The converse is immediate.
The character decomposition and the minimum-basis argument of
Proposition~\ref{prop:central-character-main} now give
\eqref{eq:relative-character-formula}.
\end{proof}

\begin{theo}\label{thm:relative-km-general}
Let $\operatorname{char}(k)\ne p$, and let
$f:\cX\to\cY$ be a locally full morphism of finite gerbes whose geometric
inertia groups are $p$-groups.  Then
\begin{equation}\label{eq:relative-km-general}
\ed_k(f;p)=\rdim_p(f).
\end{equation}
Moreover, after base change to a $p$-closure $k^{(p)}/k$, there are an
extension $K/k^{(p)}$ and an object $y\in\cY(K)$ such that
\begin{equation}\label{eq:relative-attainment}
\ed_K(\cX_y)=\ed_K(\cX_y;p)
=\rdim(\cX_y)=\rdim_p(f).
\end{equation}
\end{theo}

\begin{proof}
We first work over a field $F$ containing $\mu_p$ over which
$\band_{\cX}$ has $p$-primary descent.  Put $r=\rdim(f)$ and let
$\fC_f$ be \eqref{eq:relative-central-subband}.  Rigidify by this
central subband:
\[
q:\cX\ra\cH:=\cX\!\sslash\fC_f.
\]
Since $\fC_f\subset I_{\cX/\cY}$, the subgroup $\fC_f$ acts
trivially after applying $f$.  By the universal property of
rigidification \cite[Theorem~5.1.5]{ACV}, there is a morphism
$h:\cH\to\cY$ and a $2$-isomorphism $f\simeq hq$; we use this
$2$-isomorphism to identify $f$ with $hq$.  Choose the generic frame object
$\eta:\Spec K\to\cH$ as in
Lemma~\ref{lem:generic-object-rigidification}, put
\[
\cA_\eta=\cX\times_{\cH,\eta}\Spec K,
\quad y=h(\eta).
\]
Then $\cA_\eta$ is a gerbe banded by $\fC_{f,K}$.  Moreover,
$\eta$ is an object of the fiber $\cH_y$, and
\[
\cA_\eta
=\cX_y\times_{\cH_y,\eta}\Spec K.
\]
Thus the natural morphism $\cA_\eta\to\cX_y$ is faithful: on the
inertia group of any object it is the inclusion
$\fC_{f,K}\hookrightarrow I_{\cX/\cY,K}$.  In particular,
$\cA_\eta$ is a fiber of the morphism $\cX_y\to\cH_y$.

Because $\fC_f\simeq(\mu_p)^c$ is a split finite diagonalizable
central subband of $\cX$, the hypotheses of
Theorem~\ref{thm:nonneutral-central-index} apply to the rigidification by
$\fC_f$.  Hence, for every $\chi\in\fC_f^*$,
\[
\ind\beta_\eta(\chi)=d_\chi(\cX),
\]
where the rank ideal is taken with respect to the $\fC_f$-character
summand.  Under the present hypotheses Lemma~\ref{lem:p-power-ranks-main}
shows, more generally, that every simple object of $\Vect(\cX)$ has
$p$-power rank.  Applying the last paragraph of its proof to the
character decomposition with respect to $\fC_f$, the ranks in
$\Vect_\chi(\cX)$ are nonnegative integral combinations of powers of
$p$; their greatest common divisor is therefore their least positive
value.  Thus
\[
d_\chi(\cX)=m_\chi(\cX),
\]
and consequently
\[
\ind\beta_\eta(\chi)=d_\chi(\cX)=m_\chi(\cX).
\]
Theorem~\ref{thm:km-elementary-computation} and
Proposition~\ref{prop:relative-faithfulness-criterion} therefore yield
\[
\ed_K(\cA_\eta;p)
=\min_{\mathcal B}
\sum_{\chi\in\mathcal B}\ind\beta_\eta(\chi)
=\rdim(f)=r.
\]
The same calculation gives
$\ed_K(\cA_\eta)=\rdim(\cA_\eta)=r$.
Since $\cA_\eta$ is the fiber of $\cX_y\to\cH_y$ over $\eta$,
Lemma~\ref{lem:finite-fiber-ed} gives
$\ed_K(\cA_\eta;p)\leq\ed_K(\cX_y;p)$.  On the other hand, an
$f$-faithful bundle of rank $r$ restricts to a faithful bundle on every
fiber.  Consequently
\[
r=\ed_K(\cA_\eta;p)
\leq\ed_K(\cX_y;p)
\leq\ed_K(\cX_y)
\leq\rdim(\cX_y)
\leq r,
\]
which proves \eqref{eq:relative-attainment} over $F$.  The same upper
bound on every fiber gives
$\ed_F(f;p)=\rdim(f)$.

Now let $k$ be arbitrary and set $r=\rdim_p(f)$.
By \eqref{eq:relative-p-local-minimum}, a relative faithful bundle of
rank $r$ exists after a finite extension $L/k$ of degree prime to
$p$.  Given $K/k$ and $y\in\cY(K)$, choose a residue field $M$ of
$K\otimes_kL$ whose degree over $K$ is prime to $p$.  The restricted
bundle shows that
$\rdim_p(\cX_y)\leq r$, and
Theorem~\ref{thm:gerby-km} gives
$\ed_K(\cX_y;p)\leq r$.

Now take $F=k^{(p)}$.  This field contains $\mu_p$, and every finite
$p$-band over $F$ has $p$-primary descent.  Applying the first part over
$F=k^{(p)}$ produces the fiber in \eqref{eq:relative-attainment} and
proves the reverse inequality.
\end{proof}

\subsection{Lifting gerbes for torsors}

Let
\begin{equation}\label{eq:p-group-extension}
1\ra N\ra P'
\xrightarrow{\pi}P\ra1
\end{equation}
be an exact sequence of finite constant $p$-groups.  For a $P$-torsor
$E$ over a field $K/k$, put
\[
\operatorname{Lift}_{P'}(E)
:=\Spec K\times_{BP}BP'=[E/P'].
\]
Its objects over $L/K$ are $P'$-torsors whose quotient by $N$ is
identified with $E_L$.  Thus it is the gerbe of solutions of the finite
embedding problem determined by $E$ and \eqref{eq:p-group-extension}.
This interpretation of finite gerbes is standard; see, for example,
\cite[Section~5]{BRV}.

Define
\begin{equation}\label{eq:relative-group-rank}
\rdim_p(P',N)=
\min\{\dim(V):V\in\Rep_{k^{(p)}}(P'),
                     \ V|_N\text{ is faithful}\}.
\end{equation}

\begin{coro}
\label{cor:embedding-problem}
For the extension \eqref{eq:p-group-extension},
\[
\sup_{\substack{K/k\\E\in H^1(K,P)}}
\ed_K\bigl(\operatorname{Lift}_{P'}(E);p\bigr)
=\rdim_p(P',N).
\]
Hence $\rdim_p(P',N)$ is the optimal uniform bound for the essential
$p$-dimension of the solution gerbes of this embedding problem.
\end{coro}

\begin{proof}
Apply Theorem~\ref{thm:relative-km-general} to
$BP'\to BP$.  Its relative inertia is $N$, and
Definition~\ref{def:relative-rank} becomes
\eqref{eq:relative-group-rank}.
\end{proof}

\begin{exam}
\label{ex:cyclic-relative-rank}
Assume $\mu_p\subset k$ and
$[k(\mu_{p^2}):k]=p$.  For
\[
1\ra C_p\ra C_{p^2}
\ra C_p\ra1,
\]
the lifting gerbe of the trivial $C_p$-torsor is $BC_p$, of faithful
rank one.  Nevertheless
\[
\rdim(BC_{p^2}\to BC_p)=p.
\]
To see the asserted relative rank, let $g$ be a generator of $C_{p^2}$.
Over a separable closure, every irreducible character sends $g$ to
$\zeta_{p^2}^a$ for some $a$.  Its restriction to the kernel
$\langle g^p\rangle\simeq C_p$ is nontrivial exactly when $p\nmid a$; in
that case the character is faithful on $C_{p^2}$ and its values generate
$k(\mu_{p^2})$.  Hence such a character has Galois orbit of size
$[k(\mu_{p^2}):k]=p$.  Any $k$-representation which is nontrivial on the
kernel must contain this whole orbit after scalar extension, and therefore
has dimension at least $p$.  Conversely, on the $p$-dimensional
$k$-vector space $k(\mu_{p^2})$, let $g$ act by multiplication by a
primitive $p^2$th root of unity.  This gives a representation which is
nontrivial on the kernel, so the lower bound is attained.
Thus the trivial fiber need not attain the maximum.  The next proposition
shows that the corresponding generic $C_p$-torsor does; in this example its lifting gerbe
has essential dimension $p$, and its cyclic-algebra obstruction has index
$p$, as in \cite[Example~4.6]{KM}.

\end{exam}

\begin{prop}
\label{pro:generic-embedding-problem}
Choose a $P$-representation $W$ with a nonempty $P$-stable open subset
$U\subset W$ on which $P$ acts freely.  Put
\[
K=k(U/P),
\qquad
E_{\mathrm{gen}}=U_K.
\]
Then
\begin{equation}\label{eq:generic-embedding-attainment}
\ed_K\bigl(\operatorname{Lift}_{P'}(E_{\mathrm{gen}});p\bigr)
=\rdim_p(P',N).
\end{equation}
In particular, $E_{\mathrm{gen}}$ realizes the supremum in
Corollary~\ref{cor:embedding-problem}.

Moreover, let $F=k^{(p)}$ be a $p$-closure of $k$, put
\[
K_F=F(U_F/P),
\]
and let $E_{\mathrm{gen},F}$ be the generic fiber of
$U_F\to U_F/P$.  Then
\begin{equation}\label{eq:generic-embedding-attainment-pclosure}
\ed_{K_F}\bigl(\operatorname{Lift}_{P'}(E_{\mathrm{gen},F})\bigr)
=
\ed_{K_F}\bigl(\operatorname{Lift}_{P'}(E_{\mathrm{gen},F});p\bigr)
=
\rdim_p(P',N).
\end{equation}
\end{prop}

\begin{proof}
We first prove the stronger assertion after passage to the $p$-closure.
Thus replace $k$ by $F=k^{(p)}$.  Then $k$ is $p$-special, hence
$\mu_p\subset k$.  After this replacement write again
\[
K=k(U/P),
\qquad
E_{\mathrm{gen}}=U_K,
\]
and set
\[
\cL=[E_{\mathrm{gen}}/P'],
\quad C=N\cap Z(P')[p].
\]
For $\chi\in C^*$, restriction induces a surjection
\begin{equation}\label{eq:generic-lifting-g0}
G_0^\chi(BP')\ra G_0^\chi(\cL).
\end{equation}
Indeed, it factors as
\[
G_0^\chi(BP')
\xrightarrow{\sim}G_0^\chi([W/P'])
\ra G_0^\chi([U/P'])
\ra G_0^\chi(\cL).
\]
The first arrow is homotopy invariance and the second is localization.
For the third, spread a coherent sheaf from the generic fiber over the
inverse image of some nonempty open subset of $U/P$, and apply localization,
exactly as in the proof of
Lemma~\ref{lem:generic-central-g0-surjection}.  Since restriction preserves rank,
\eqref{eq:generic-lifting-g0} gives
\begin{equation}\label{eq:generic-lifting-rank-ideals}
I_\chi(BP')=I_\chi(\cL).
\end{equation}

The extension $k(U)/K$ is $P$-Galois and neutralizes $\cL$.  Hence
the descent action on its geometric band $N$ contains the full
conjugation action of $P=P'/N$.  The kernel of a vector bundle on
$\cL$ is therefore represented by a subgroup of $N$ normal in
$P'$.  Every nontrivial such subgroup meets $C$.  Thus faithfulness on
$\cL$, and relative faithfulness for $BP'\to BP$, are both equivalent
to faithfulness on $C$.

In both character categories, every simple object has rank a power of $p$; hence the associated rank ideals are realized by least blocks, as in the proof of Lemma~\ref{lem:p-power-ranks-main}.  Using Equation \eqref{eq:generic-lifting-rank-ideals} together with the two minimum basis formulas, we obtain
\[
\rdim(\cL)=\rdim(BP'\to BP)=\rdim_p(P',N).
\]
Moreover, since the band of $\cL$ admits $p$-primary descent, Theorem~\ref{thm:gerby-km-special} yields
\[
\ed_K(\cL)=\ed_K(\cL;p)=\rdim(\cL)=\rdim_p(P',N).
\]
Thus \eqref{eq:generic-embedding-attainment-pclosure} follows.

We now return to the original base field $k$.  Put
\[
\cL_0=\operatorname{Lift}_{P'}(E_{\mathrm{gen}}).
\]
After base change to $F=k^{(p)}$, the generic torsor
$E_{\mathrm{gen}}$ becomes $E_{\mathrm{gen},F}$ over $K_F$, and hence
\[
(\cL_0)_{K_F}\simeq
\operatorname{Lift}_{P'}(E_{\mathrm{gen},F}).
\]
By Lemma~\ref{lem:ed-base-change} and the equality just proved,
\[
\ed_K(\cL_0;p)
\ge
\ed_{K_F}\bigl((\cL_0)_{K_F};p\bigr)
=
\rdim_p(P',N).
\]
On the other hand, Corollary~\ref{cor:embedding-problem} gives
\[
\ed_K(\cL_0;p)\le \rdim_p(P',N).
\]
Therefore
\[
\ed_K\bigl(\operatorname{Lift}_{P'}(E_{\mathrm{gen}});p\bigr)
=\rdim_p(P',N),
\]
which is \eqref{eq:generic-embedding-attainment}.  The assertion that
$E_{\mathrm{gen}}$ realizes the supremum now follows from
Corollary~\ref{cor:embedding-problem}.
\end{proof}

\begin{rema}[Relation to L\"otscher's question]
L\"otscher asked whether, for a normal subgroup $C\triangleleft G$, the
largest essential $p$-dimension of the lifting gerbes attached to
$G/C$-torsors is already realized by the generic torsor obtained from a
free open subset of a $G/C$-representation; he proved this
when $C$ is central and isomorphic to a power of $\mu_p$
\cite[Remark~4.4]{LotscherFiber}.  Taking $G=P'$ and $C=N$, the preceding
proposition gives an affirmative answer for arbitrary normal subgroups of finite
$p$-groups.
\end{rema}

\section{Compression and quotient singularities}\label{sec:compression-singularities}

For an algebraic group scheme, essential dimension has a classical
interpretation in terms of compression.  If $\fG$ acts generically freely on
an integral variety $Y$, a $\fG$-compression of $Y$ is a dominant
$\fG$-equivariant rational map
\[
Y\dashrightarrow X
\]
to another generically free $\fG$-variety.  Starting from a generically free
linear representation, the essential dimension of $\fG$ is the minimum of
$\dim(X)-\dim(\fG)$ among such compressions; for finite groups this is simply
the minimum of $\dim(X)$.  This is the original geometric point of view on
essential dimension; see \cite{BuhlerReichstein,ReichsteinED}.  We have
already used it in Remark~\ref{rem:localization-necessary}, where the rational
map from a faithful representation to its projectivization gives a
one dimensional compression.

This description depends on having a group scheme acting on a variety.  But for a
non-neutral gerbe $\cG$ there is no direct analogue. The purpose of this section is to describe a
local substitute.  Bresciani and Vistoli associate to a tame quotient
singularity $(S,s)$ a finite gerbe $\cG_{(S,s)}$, called its
\emph{fundamental gerbe}: it is the residual gerbe over $s$ in the minimal
stack $\widehat S\to S$ \cite[Proposition~6.1 and Definition~6.2]{Bresciani_Vistoli_2024}.
We shall measure how small a quotient singularity can be while retaining a
prescribed gerbe as its fundamental gerbe.

\subsection{Quotient-compression dimension}

Recall that, we say a variety of finite type over $k$ has tame quotient singularity, if \'etale locally, it is isomorphic to $U/G$, where $U$ is a smooth scheme and $G$ is a finite group of order invertible in $k$. We refer to \cite[\S 6]{Bresciani_Vistoli_2024} for a general discussion.

Let $\cG/k$ be a finite gerbe whose geometric inertia has order invertible in
$k$. We call a vector bundle $\cE$ on $\cG$ \emph{small} if no nontrivial geometric inertia element acts as
a pseudoreflection. Put
\[
\operatorname{srdim}_k(\cG)
=
\min\{\rk(\cE):\cE\text{ is faithful and small}\}.
\]
The following construction is the bridge from vector bundles on a gerbe to
quotient singularities.

\begin{prop}\label{prop:gerbe-to-quotient-singularity}
Let $\cG/k$ be as above and let $\cE$ be a faithful small vector bundle of
rank $d$.  Let
\[
\pi:\cE\rightarrow \bfE
\]
be its coarse moduli morphism.  Then $\bfE$ is an affine $k$-scheme, the zero
section induces a point $0\in\bfE(k)$, and $(\bfE,0)$ is a $d$-dimensional
tame quotient singularity.  Moreover,
\[
\widehat{\bfE}\simeq \cE,
\quad
\cG\simeq\cG_{(\bfE,0)},
\]
where the second isomorphism is induced by the zero section.
\end{prop}

\begin{proof}
After passing to a separable closure
$k_s/k$, choose an object $\bar x\in\cG(k_s)$ and put
$P=\Aut(\bar x)$.  Let $V$ be the fiber of $\cE$ at $\bar x$.  Then $V$ is a
$d$-dimensional representation of $P$, and
\[
\cG_{k_s}\simeq BP,
\quad
\cE_{k_s}\simeq[V/P].
\]
The representation $V$ is faithful, hence the generic stabilizer of
$[V/P]$ is trivial.  Its coarse moduli space is the affine quotient $V/P$.
Formation of coarse moduli spaces for tame stacks commutes with extension of
the base field and since affineness descends fpqc-locally, $\bfE$ is an affine $k$-scheme. The zero section $\cG\hookrightarrow\cE$ induces a rational point $0\in\bfE(k)$ and the $d$-dimensional scheme $\bfE$ has tame quotient singularities by
\cite[Proposition~6.1]{Bresciani_Vistoli_2024}.

It remains to identify the minimal stack and its residual gerbe.  The
smallness assumption says that $P\subseteq\GL(V)$ contains no
pseudoreflections. So we have
\[
[V/P]\simeq\widehat{(V/P)};
\]
by \cite[Lemma~6.3]{Bresciani_Vistoli_2024}. Again, using the fact that formation of the minimal
stack commutes with base change, the canonical morphism
$\cE\to\widehat{\bfE}$ induced by the universal property of minimal resolution consequently becomes an isomorphism after base change
to $k_s$, and hence is an isomorphism over $k$.

Finally,
\[
\cG_{(\bfE,0)}
=
\bigl(\Spec k\times_{\bfE}\widehat{\bfE}\bigr)_{\mathrm{red}}.
\]
Via $\widehat{\bfE}\simeq\cE$, the zero section induces a morphism
$\cG\to\cG_{(\bfE,0)}$.  After base change to $k_s$, this is
\[
BP\longrightarrow
\bigl(\{0\}\times_{V/P}[V/P]\bigr)_{\mathrm{red}}.
\]
For a finite linear group, the quotient $V\to V/P$ separates geometric
orbits, so the reduced fiber over the image of the origin is $\{0\}$.  The
target is therefore $[0/P]=BP$.  Thus the displayed morphism is an
isomorphism after base change to $k_s$, and hence
$\cG\simeq\cG_{(\bfE,0)}$ over $k$.
\end{proof}

\begin{defi}\label{def:qcdim}
Let $\cG/k$ be a finite gerbe whose geometric inertia has order invertible in
$k$.  Its \emph{quotient-compression dimension} is
\[
\operatorname{qcdim}_k(\cG)
=
\min\left\{\dim S:\
\begin{array}{l}
(S,s)\text{ is a tame quotient singularity over }k,\\
\cG_{(S,s)}\simeq\cG
\end{array}
\right\}.
\]
If no such realization exists, we set $\operatorname{qcdim}_k(\cG)=\infty$.
\end{defi}

Thus $\operatorname{qcdim}$ is a local realization invariant.  In the neutral
case $\cG=BG$, it asks for the smallest dimension of a quotient singularity
whose local structure retains the whole group $G$; for a non-neutral
gerbe the definition is intrinsic and does not require a global group scheme band.

\subsection{Faithful rank and essential dimension}

The preceding construction gives one inequality between quotient-compression
dimension and small faithful rank.  The reverse inequality is supplied by the
normal bundle of the fundamental gerbe.

\begin{prop}\label{prop:qcdim-small-rank}
For every finite gerbe $\cG/k$ whose geometric inertia has order invertible in
$k$,
\[
\operatorname{qcdim}_k(\cG)=\operatorname{srdim}_k(\cG).
\]
In particular, every such gerbe is the fundamental gerbe of a tame quotient
singularity.
\end{prop}

\begin{proof}
Proposition~\ref{prop:gerbe-to-quotient-singularity} gives
$\operatorname{qcdim}_k(\cG)\le\operatorname{srdim}_k(\cG)$.
Conversely, suppose that $(S,s)$ is a $d$-dimensional tame quotient
singularity with fundamental gerbe $\cG$.  Let $\widehat S\to S$ be its
minimal stack.  The normal bundle
\[
N_{\cG/\widehat S}
\]
has rank $d$, and after passage to a separable closure it is the linear
representation appearing in the local description of the singularity.  That
representation is faithful and has no pseudoreflections; see
\cite[Corollary~6.4 and Section~6.4]{Bresciani_Vistoli_2024}.  Thus
$N_{\cG/\widehat S}$ is a faithful small vector bundle and
$\operatorname{srdim}_k(\cG)\le d$.

Finally, finite gerbes admit faithful vector bundles.  If $\cE$ is faithful,
then
\[
\cE\oplus\det(\cE)
\]
is faithful and small.  Indeed, if an inertia element is already nontrivial in
at least two directions on $\cE$, there is nothing to prove; if it is a
pseudoreflection, with eigenvalues $1,\ldots,1,\lambda$ and $\lambda\ne1$,
then it also acts nontrivially by $\lambda$ on $\det(\cE)$.  Hence the minimum
above is finite, and the claimed realization follows.
\end{proof}

The determinant argument also compares this local compression invariant with
the faithful rank studied throughout the paper.

\begin{coro}\label{cor:qcdim-rdim}
For every finite gerbe $\cG/k$ whose geometric inertia has order invertible in
$k$,
\[
\rdim(\cG)
\le
\operatorname{qcdim}_k(\cG)
\le
\rdim(\cG)+1.
\]
\end{coro}

The possible extra dimension has a concrete meaning: ordinary faithful rank
does not see pseudoreflections, while the minimal stack construction removes
the inertia generated by pseudoreflections. Thus a smallest faithful bundle
need not itself preserve the whole gerbe as the fundamental gerbe of the
coarse moduli space. Adding its determinant removes this obstruction at the cost
of at most one dimension.

To compare with essential dimension at $p$, define the prime local version by
\[
\operatorname{qcdim}_p(\cG)
=
\min_{\substack{L/k\ \mathrm{finite}\\ p\nmid[L:k]}}
\operatorname{qcdim}_L(\cG_L).
\]
The construction above commutes with base change, so Corollary~\ref{cor:qcdim-rdim}
gives
\[
\rdim_p(\cG)
\le
\operatorname{qcdim}_p(\cG)
\le
\rdim_p(\cG)+1.
\]
Combining this with Theorem~\ref{thm:gerby-km} yields the promised
compression interpretation of the main theorem.

\begin{coro}\label{thm:qcdim-ed}
Let $p\ne\operatorname{char}(k)$ and let $\cG/k$ be a finite gerbe whose
geometric inertia groups are $p$-groups.  Then
\[
\ed_k(\cG;p)
\le
\operatorname{qcdim}_p(\cG)
\le
\ed_k(\cG;p)+1.
\]
\end{coro}

\begin{proof}
By Theorem~\ref{thm:gerby-km},
$\ed_k(\cG;p)=\rdim_p(\cG)$.  Substitute this equality into the preceding
faithful rank inequalities.
\end{proof}

\bibliographystyle{amsplain}
\bibliography{main}

\end{document}